\documentclass[11pt,a4paper]{amsart}
\usepackage{amsmath, amssymb}
\usepackage{amsfonts}
\usepackage[arrow,matrix,curve,cmtip,ps]{xy}
\usepackage{amsthm}
\usepackage{amscd}
\usepackage{tikz}
\usepackage{wrapfig}
\usepackage{hyperref}

\allowdisplaybreaks

\theoremstyle{remark}
\newtheorem{theorem}{{Theorem}}[subsection]
\newtheorem{lemma}[theorem]{Lemma}
\newtheorem{proposition}[theorem]{Proposition}

\newtheorem{remark}[theorem]{Remark}
\newtheorem{definition}[theorem]{Definition}

\numberwithin{equation}{subsection}
\numberwithin{theorem}{subsection}
\newcount\colveccount
\newcommand*\colvec[1]{
        \global\colveccount#1
        \begin{pmatrix}
        \colvecnext
}
\def\colvecnext#1{
        #1
        \global\advance\colveccount-1
        \ifnum\colveccount>0
                \\
                \expandafter\colvecnext
        \else
                \end{pmatrix}
        \fi
}

\newcommand{\R}{\mathbb{R}}

\newcommand{\defeq}{\mathrel{\mathop:}=}

\begin{document}
\title[The Pressure Metric on the Margulis Multiverse]{The Pressure Metric on the Margulis Multiverse}

\author{Sourav Ghosh}

\address{Department of Mathematics \\ Universit\'e Paris Sud \\ Orsay 91400 \\ France}
\email{sourav.ghosh@math.u-psud.fr}

\thanks{The research leading to these results has received funding from the European Research Council under the {\em European Community}'s seventh Framework Programme (FP7/2007-2013)/ERC {\em  grant agreement}}

\date{\today}



\begin{abstract}
This paper defines the pressure metric on the Moduli space of Margulis spacetimes without cusps and shows that it is positive definite on the constant entropy sections. It also demonstrates an identity regarding the variation of the cross-ratios.

\end{abstract}

\maketitle
\tableofcontents
\pagebreak

\section{Introduction}

In \cite{marg1} and \cite{marg2} Margulis had shown that a non-abelian free group $\Gamma$ with finitely many generators $n$ can act freely and properly as affine transformations on the affine three space $\mathbb{A}$ such that the linear part of the affine action is discrete. In such a case we call the resulting quotient manifold a $\textit{Margulis}$ $\textit{spacetime}$. 

Margulis spacetimes have been studied extensively by Abels--Margulis--Soifer \cite{hgg1}\cite{hgg2}, Charette--Drumm \cite{cd}, Charette--Goldman--Jones \cite{jones}, Choi--Goldman \cite{cg}, Danciger--Gu\'eritaud--Kassel \cite{dgk}\cite{dgk2}, Drumm \cite{D}\cite{D2}, Drumm--Goldman \cite{dg1}\cite{D1}, Fried--Goldman \cite{fried}, Goldman \cite{minko}, Goldman--Labourie \cite{geodesic}, Goldman--Labourie--Margulis \cite{labourie invariant}, Goldman--Margulis \cite{vari}, Kim \cite{kim} and Smilga \cite{smil1}\cite{smil2}\cite{smil3}.

In this paper we will only consider Margulis spacetimes which have no cusps, that is, the linear part of the affine action contains no parabolic elements.  Here we mention that Margulis spacetimes with cusps were shown to exist by Drumm\cite{D2}.

Moreover, in \cite{fried} Fried--Goldman showed that if $\Gamma$ acts on $\mathbb{A}$ as affine transformations giving rise to a Margulis spacetime then a conjugate of the linear part of the action of $\Gamma$ is a subgroup of $\mathsf{SO}^0(2,1)\subset\mathsf{GL}(\R^3)$. Therefore, we can think of Margulis spacetimes as conjugacy classes $[\rho]$ of injective homomorphisms
\begin{align*}
\rho: \Gamma \longrightarrow \mathsf{SO}^0(2,1)\ltimes\mathbb{R}^3.
\end{align*}
We denote the moduli space of Margulis spacetimes with no cusps by $\mathcal{M}$. In \cite{labourie invariant} Goldman--Labourie--Margulis showed  that $\mathcal{M}$ is an open subset of the representation variety $\mathsf{Hom}(\Gamma,\mathsf{SO}^0(2,1)\ltimes\mathbb{R}^3)/\sim$ where the conjugacy is in $\mathsf{SO}^0(2,1)\ltimes\mathbb{R}^3$. Therefore $\mathcal{M}$ is an analytic manifold. Also we know from \cite{me} that the homomorphisms giving rise to Margulis spacetimes are $\textit{Anosov}$.

In this paper, we will use the $\textit{metric}$ $\textit{Anosov}$ property from section 3 of \cite{me} and the theory of thermodynamical formalism (as appeared in section 3 of \cite{pressure metric}) developed by Bowen, Bowen--Ruelle, Parry--Pollicott, Pollicott and Ruelle and others in \cite{Bow}, \cite{BR}, \cite{PP}, \cite{P}, \cite{R} to define the $\textit{entropy}$ and $\textit{intersection}$. Now using the fact that a representation giving rise to a Margulis spacetime is Anosov, we go on to show that the entropy and intersection vary analytically over $\mathcal{M}$. Moreover, we define the pressure metric on $\mathcal{M}$ and study its properties. In particular, we prove the following theorems:

\begin{theorem}\label{mainthm}
Let $\mathcal{M}_k$ be a constant entropy section of the analytic manifold $\mathcal{M}$ with entropy $k$ and let $\mathtt{P}$ be the pressure metric on $\mathcal{M}$. Then $(\mathcal{M}_k,\left.\mathtt{P}\right|_{\mathcal{M}_k})$ is an analytic Riemannian manifold.
\end{theorem}

\begin{theorem}\label{m2}
The pressure metric $\mathtt{P}$ has signature $(\dim(\mathcal{M})-1,0)$ over the moduli space $\mathcal{M}$.
\end{theorem}
We call the constant entropy sections of the analytic manifold $\mathcal{M}$ as the $\textit{Margulis}$ $\textit{multiverses}$.\\

The study of pressure metric in the context of representation varieties was started by McMullen and Bridgeman--Taylor respectively in \cite{M}, \cite{BT}. McMullen formulated the Weil--Petersson metric on the Teichm\"uller Space as a Pressure metric. The result was generalised to the quasi-Fuchsian case by Bridgeman--Taylor in \cite{BT}. In \cite{B} the pressure metric was further studied by Bridgeman in the context of the semisimple Lie group $\mathsf{SL}(2,\mathbb{C})$. Recent results by Bridgeman--Canary--Labourie--Sambarino in \cite{pressure metric} extend it in the context of any semisimple Lie group. In this paper, we will study the case where the Lie group in question is $\mathsf{SO}^0(2,1)\ltimes\mathbb{R}^3$, a non-semisimple Lie group.

Moreover, in the process of proving theorem \ref{mainthm} and \ref{m2} we will also come up with a formula for the variation of the cross-ratios in section \ref{cross}.\\

\textbf{Acknowledgments:} I would like to express my gratitude towards my advisor Prof. Francois Labourie for his guidance. I would like to thank Dr. Andres Sambarino for the many helpful discussions that we had. I would also like to thank the organisers of the Aarhus conference on pressure metric and DMS program at MSRI for giving me the opportunity to discuss with Prof. Martin Bridgemann, Prof. Richard Canary, Prof. Olivier Guichard, Prof. Mark Pollicott and Dr. Maria Beatrice Pozzetti.

\section{Background}

\subsection{Hyperbolic geometry}

Let $\left(\mathbb{R}^{2,1}, \langle\mid\rangle\right)$ be a Minkowski spacetime. The quadratic form which corresponds to the metric $\langle\mid\rangle$ is given by 
\begin{align}\label{lorentz}
\mathcal{Q} \defeq 
\begin{pmatrix}
    1 & 0 & 0\\
    0 & 1 & 0\\
    0 & 0& -1\\
\end{pmatrix}.
\end{align}
We denote the group of linear transformations preserving the metric $\langle\mid\rangle$ on $\mathbb{R}^{2,1}$ by $\mathsf{SO}(2,1)$. Moreover, let $\mathsf{SO}^{0}(2,1)$ be the connected component of $\mathsf{SO}(2,1)$ which contains the identity element.

Now let us consider the following spaces
\begin{align*} 
\mathsf{S}^{k} \defeq \lbrace v \in \mathbb{R} \mid \langle v,v \rangle \ = k \rbrace
\end{align*}
where $k\in\mathbb{R}$. We notice that $\mathsf{S}^{-1}$ has two components. Let us denote the component of $\mathsf{S}^{-1}$ which contains ($0,0,1$)$^{t}$ as $\mathbb{H}$. The space $\mathbb{H}$ as a submanifold of $\left(\mathbb{R}^{2,1}, \langle\mid\rangle\right)$ has a constant negative curvature of $-1$ for the restriction of the metric $\langle\mid\rangle$. It is called the $\textit{hyperboloid}$ $\textit{model}$ $\textit{of}$ $\textit{hyperbolic}$ $\textit{geometry}$. Let us denote the unit tangent bundle of $\mathbb{H}$ by $\mathsf{U}\mathbb{H}$. Now we consider the map
\begin{align} \label{theta}
\Theta : \mathsf{SO}^{0}(2,1) &\longrightarrow \mathsf{U}\mathbb{H}\\
\notag g &\longmapsto \left(g(0,0,1)^t, g(0,1,0)^t\right).
\end{align}
The group $\mathsf{SO}^{0}(2,1)$ can be analytically identified with $\mathsf{U}\mathbb{H}$ via the map $\Theta$. Let $\tilde{\phi}_t$ be the geodesic flow on $\mathsf{U}\mathbb{H} \cong \mathsf{SO}^{0}(2,1)$. We note that
\begin{align}
\tilde{\phi}_t \colon  \mathsf{SO}^0(2,1) &\longrightarrow  \mathsf{SO}^0(2,1)\\
\notag g &\longmapsto g\begin{pmatrix}
    1 & 0 & 0\\
    0 & \cosh(t) & \sinh(t)\\
    0 & \sinh(t)& \cosh(t)\\
\end{pmatrix}.
\end{align}

Now we define the $\textit{neutral section}$ ${\nu}$ as follows:
\begin{align}
{\nu} \colon  \mathsf{SO}^0(2,1) &\longrightarrow  \mathsf{S}^{1}\\
\notag g &\longmapsto g(1,0,0)^t,
\end{align}
The neutral section is invariant under the geodesic flow $\tilde{\phi}_t$ and moreover for all $g, h\in\mathsf{SO}^0(2,1)$ we have
\[{\nu}(h.g) = h.{\nu}(g).\]

Let $\partial_{\infty}\mathbb{H}$ be the boundary at infinity of $\mathbb{H}$. As the neutral section is invariant under the geodesic flow it gives rise to an analytic map,
\begin{align}
{\nu} \colon  \partial_{\infty}\mathbb{H} \times \partial_{\infty}\mathbb{H} \setminus \Delta \longrightarrow  \mathsf{S}^{1}.
\end{align}
For any hyperbolic element $\gamma\in\mathsf{SO}^0(2,1)$ acting on $\mathbb{H}$ and for any $x\in\mathbb{H}$ we define
\[{\gamma}^\pm \defeq \lim_{n \to \pm\infty}\gamma^n x.\]
We note that $\gamma^\pm\in\partial_{\infty}\mathbb{H}$  is well defined as the limit is independent of the point $x\in\mathbb{H}$.  We observe that
\begin{align}
\gamma{\nu}({\gamma^-},{\gamma}^+) = {\nu}({\gamma^-},{\gamma}^+),
\end{align}
that is, ${\nu}({\gamma^-},{\gamma}^+)$ is an eigenvector of $\gamma$ with eigenvalue 1. Moreover for $a, b, c, d$ in $\partial_{\infty}\mathbb{H}$ with $a\neq d$ and $b\neq c$ let
\begin{align}
\mathsf{b}(a,b,c,d) \defeq \frac{1}{2}\left( 1 + \langle{\nu}(a,d)\mid{\nu}(b,c)\rangle\right).
\end{align}
Now we list a few identities satisfied by ${\nu}$ and $\mathsf{b}$:
\begin{align}
&{\nu}(a,b) + {\nu}(b,a) = 0, \label{b1}\\
&\langle {\nu}(a,b)\mid {\nu}(a,c)\rangle = 1,\\
&\mathsf{b}(d,b,c,a){\nu}(a,b)+ \mathsf{b}(a,b,c,d){\nu}(a,c) = {\nu}(a,d), \label{b2}\\
&\mathsf{b}(a,b,c,d)=\mathsf{b}(b,a,d,c)=\mathsf{b}(d,c,b,a), \\
&\mathsf{b}(a,b,c,d)+\mathsf{b}(d,b,c,a)=1, \label{b3}\\
&\mathsf{b}(a,w,c,d)\mathsf{b}(w,b,c,d)=\mathsf{b}(a,b,c,d),
\end{align}
where $a,b,c,d,w$ are pairwise distinct points in $\partial_{\infty}\mathbb{H}$.
We notice that $\mathsf{b}$ is the classical $\textit{cross ratio}$.

Let $\Gamma\subset\mathsf{SO}^0(2,1)$ be a non-abelian free group with finitely many generators acting as a Schottky group on $\mathbb{H}$. We denote the surface $\Gamma \backslash \mathbb{H}$ by $\Sigma$ and the unit tangent bundle of $\Sigma$ by $\mathsf{U}\Sigma$. We observe that 
\[\Gamma \backslash \mathsf{U}\mathbb{H}\cong\mathsf{U}\Sigma.\] 
We note that the flow $\tilde{\phi}$ on $\mathsf{U}\mathbb{H}$ induces a flow $\phi$ on $\mathsf{U}\Sigma$. Let $\Lambda_{\infty}\Gamma$ be the $\textit{limit set}$ of the group $\Gamma$. We recall that
\[\Lambda_{\infty}\Gamma=\left.\overline{\Gamma .x}\right.\backslash\Gamma .x\subset\partial_\infty\mathbb{H}\]
where $x$ is any point in $\mathbb{H}$ and $\overline{\Gamma .x}\subset\mathbb{H}\cup\partial_\infty\mathbb{H}$ denote the closure of the space
\[\Gamma .x\defeq\{\gamma.x\mid\gamma\in\Gamma\}\subset\mathbb{H}.\] 
We also recall that as $\Gamma$ is Schottky, the limit set $\Lambda_{\infty}\Gamma$ is a cantor set.

Let us denote the space of all non-wandering points of the geodesic flow $\phi$ on $\mathsf{U}\Sigma$ by $\mathsf{U}_{\hbox{\tiny $\mathrm{rec}$}}\Sigma$ and the lift of the space $\mathsf{U}_{\hbox{\tiny $\mathrm{rec}$}}\Sigma$ in $\mathsf{U}\mathbb{H}$ by $\mathsf{U}_{\hbox{\tiny $\mathrm{rec}$}}\mathbb{H}$. The space $\mathsf{U}_{\hbox{\tiny $\mathrm{rec}$}}\Sigma$ is compact. Furthermore, we note that
\begin{align*}
\mathsf{U}_{\hbox{\tiny $\mathrm{rec}$}}\mathbb{H} &= \left\lbrace (x,v) \in \mathsf{U}\mathbb{H} \mid  \lim_{t \to \pm\infty}\tilde{\phi}^1_t (x,v) \in \Lambda_{\infty}\Gamma\right\rbrace\\
&\cong\left(\Lambda_{\infty}\Gamma\times\Lambda_{\infty}\Gamma\setminus\Delta \right)\times\mathbb{R}
\end{align*}
where $\tilde{\phi}_t(x,v) = (\tilde{\phi}^1_t(x,v), \tilde{\phi}^2_t(x,v))$ and $\Delta\defeq\{(x,x)\mid x\in \Lambda_{\infty}\Gamma\}$.

\subsection{Margulis spacetimes}

Let $\Gamma$ be a non-abelian free group with finitely many generators $n$ and let $\mathbb{A}$ be the affine three space whose underlying vector space is given by $\mathbb{R}^3$. Now we consider an injective homomorphism $\rho$ of $\Gamma$ into the affine linear group $\mathsf{Aff}(\mathbb{A})\cong\mathsf{GL}(\mathbb{R}^3)\ltimes\mathbb{R}^3$, that is,
\begin{align*}
\rho:\Gamma&\longrightarrow\mathsf{GL}(\mathbb{R}^3)\ltimes\mathbb{R}^3\\
\gamma&\longmapsto\left(\mathtt{L}_\rho(\gamma),\mathtt{u}_\rho(\gamma)\right).
\end{align*}
We respectively call $\mathtt{L}_\rho$ the $\textit{linear}$ $\textit{part}$ of $\rho$ and $\mathtt{u}_\rho$ the $\textit{translation}$ $\textit{part}$ of $\rho$.

In \cite{marg1} and \cite{marg2}, Margulis had shown that there exists $\rho$ such that $\rho(\Gamma)$ acts freely and properly on the affine space $\mathbb{A}$ with $\mathtt{L}_\rho(\Gamma)$ being discrete. In such a case we call the quotient manifold $\mathsf{M}_\rho\defeq\rho(\Gamma)\backslash\mathbb{A}$ a $\textit{Margulis}$ $\textit{spacetime}$.

If $\rho$ is an injective homomorphism of $\Gamma$ into $\mathsf{GL}(\mathbb{R}^3)\ltimes\mathbb{R}^3$ giving rise to a Margulis spacetime then using a result proved by Fried--Goldman in \cite{fried} it follows that a conjugate of $\mathtt{L}_\rho(\Gamma)$ is a subgroup of $\mathsf{SO}^0(2,1)$. Therefore without loss of generality we can denote a Margulis spacetime by a conjugacy class of homomorphisms
\[\rho:\Gamma\longrightarrow\mathsf{G}\defeq\mathsf{SO}^0(2,1)\ltimes\mathbb{R}^3.\]
In this paper we will only consider Margulis spacetimes $[\rho]$ such that $\mathtt{L}_\rho(\Gamma)$ contains no parabolic elements.

Let $\mathsf{M}_\rho\defeq\rho(\Gamma)\backslash\mathbb{A}$ be a Margulis spacetime such that $\mathtt{L}_\rho(\Gamma)$ contains no parabolic elements. Then the action of $\mathtt{L}_\rho(\Gamma)$ on $\mathbb{H}$ is Schottky \cite{D2}. Hence $\Sigma_{\mathtt{L}_\rho}\defeq\mathtt{L}_\rho(\Gamma)\backslash\mathbb{H}$ is a non-compact surface with no cusps. 

Now let us denote the tangent bundle of $\mathsf{M}_\rho$ by $\mathsf{T}\mathsf{M}_\rho$. We note that $\mathsf{T}\mathsf{M}_\rho$ carries a Lorentzian metric $\langle\mid\rangle$ as $\mathtt{L}_\rho(\Gamma)\subset\mathsf{SO}^0(2,1)$. Moreover, we consider the following subspace
\[\mathsf{U}\mathsf{M}_\rho\defeq\{(X,v)\in \mathsf{T}\mathsf{M}_\rho\mid\langle v\mid v\rangle_X=1\}.\]
We note that 
\[\mathsf{U}\mathsf{M}_\rho\cong\rho(\Gamma)\backslash\mathsf{U}\mathbb{A}\] where $\mathsf{U}\mathbb{A}\defeq\mathbb{A}\times\mathsf{S}^{1}$. Let us denote the induced flow on $\mathsf{U}\mathsf{M}_\rho$ coming from the geodesic flow $\tilde{\Phi}$ on $\mathsf{T}\mathbb{A}$ by $\Phi$. Note that for any real number $t$,
\begin{align}
\tilde{\Phi}_t : \mathsf{T}\mathbb{A} &\longrightarrow \mathsf{T}\mathbb{A}\\
\notag(X,v)&\longmapsto (X+tv,v).
\end{align}
Now let $\mathsf{U}_{\hbox{\tiny $\mathrm{rec}$}}\mathsf{M}_\rho$ be the space of all non-wandering points of the flow $\Phi$ on $\mathsf{U}\mathsf{M}_\rho$ and also let $\mathsf{U}^\rho_{\hbox{\tiny $\mathrm{rec}$}}\mathbb{A}$ be the lift of $\mathsf{U}_{\hbox{\tiny $\mathrm{rec}$}}\mathsf{M}_\rho$ into the space $\mathsf{U}\mathbb{A}$. Moreover, we denote the lift of $\mathsf{U}_{\hbox{\tiny $\mathrm{rec}$}}\Sigma_{\mathtt{L}_\rho}$ in $\mathsf{U}\mathbb{H}$ by $\mathsf{U}^\rho_{\hbox{\tiny $\mathrm{rec}$}}\mathbb{H}$.

\begin{theorem}{[Goldman--Labourie--Margulis]}{(see \cite{labourie invariant})} \label{N}
Let $\rho$ be an injective homomomorphism of $\Gamma$ into $\mathsf{G}$ which gives rise to a Margulis spacetime and let $\mathtt{L}_\rho(\Gamma)$ contain no parabolic elements. Then there exists a positive H\"older continuous function
\[{f}_\rho: \mathsf{U}^\rho_{\hbox{\tiny $\mathrm{rec}$}}\mathbb{H}\longrightarrow\mathbb{R}\]
and a map 
\[{N}_\rho: \mathsf{U}^\rho_{\hbox{\tiny $\mathrm{rec}$}}\mathbb{H}\longrightarrow\mathbb{A}\]
such that for all $\gamma\in\Gamma$, $g\in\mathsf{U}^\rho_{\hbox{\tiny $\mathrm{rec}$}}\mathbb{H}$ and $t\in\mathbb{R}$ we have
\begin{enumerate}
\item ${f}_\rho\circ\mathtt{L}_\rho(\gamma) = {f}_\rho$,
\item ${N}_\rho\circ\mathtt{L}_\rho(\gamma) = \rho(\gamma){N}_\rho$,
\item ${N}_\rho(\tilde{\phi_t}g) = {N}_\rho(g) + \left(\int\limits_{0}^{t}{f}_\rho(\tilde{\phi_s}(g))ds\right){\nu}(g).$
\end{enumerate}
\end{theorem}
The map ${N}_\rho$ is called a $\textit{neutralised section}$. Moreover, the following result was proved in \cite{geodesic} by Goldman--Labourie:
\begin{theorem}{[Goldman--Labourie]} \label{commute}
Let $\rho$ be an injective homomomorphism of $\Gamma$ into $\mathsf{G}$ which gives rise to a Margulis spacetime and let $\mathtt{L}_\rho(\Gamma)$ contains no parabolic elements. Also let $\mathtt{N}_\rho\defeq(N_\rho,\nu)$ where ${N}_\rho$ is a neutralised section. Then there exists an injective map $\hat{\mathtt{N}}_\rho$ such that the following diagram commutes,
\[
\begin{CD}
\mathsf{U}^\rho_{\hbox{\tiny $\mathrm{rec}$}}\mathbb{H} @> \mathtt{N}_\rho >> \mathsf{U}\mathbb{A} \\
@ V{\pi} VV @ VV {\pi} V \\
{\mathsf{U}_{\hbox{\tiny $\mathrm{rec}$}}\Sigma_{\mathtt{L}_\rho}} @> {\hat{\mathtt{N}}_\rho} >> {\mathsf{U} \mathsf{M}}_\rho. \\
\end{CD}
\]
Moreover, $\hat{\mathtt{N}}_\rho$ gives an orbit equivalent H\"older homeomorphism between $\mathsf{U}_{\hbox{\tiny $\mathrm{rec}$}}\Sigma_{\mathtt{L}_\rho}$ and $\mathsf{U}_{\hbox{\tiny $\mathrm{rec}$}}\mathsf{M}_\rho$.
\end{theorem}

We note that, if $\rho$ is an injective homomorphism of $\Gamma$ into $\mathsf{G}$ then $\mathtt{L}_\rho$ is an injective homomorphism of $\Gamma$ into $ \mathsf{SO}^0(2,1)$ and $\mathtt{u}_\rho$ satisfies the $\textit{cocycle}$ $\textit{identity}$, that is, for all $\gamma_1,\gamma_2\in\Gamma$
\[\mathtt{u}_\rho(\gamma_1.\gamma_2)=\mathtt{L}_\rho(\gamma_1)\mathtt{u}_\rho(\gamma_2)+\mathtt{u}_\rho(\gamma_1).\] 
Let us denote the space of all injective homomorphisms from the group $\Gamma$ into a Lie group $G$ by $\mathsf{Hom}(\Gamma,G)$ and the space of cocycles by $\mathsf{Z}^1({\mathtt{L}_\rho}(\Gamma), \mathbb{R}^3)$. We denote the space of all homomorphisms $\rho$ in $\mathsf{Hom}(\Gamma, \mathsf{G})$ such that $\rho(\Gamma)$ acts freely and properly on $\mathbb{A}$ and $\mathtt{L}_\rho(\Gamma)$ is discrete containing no parabolic elements by $\mathsf{Hom}_{\hbox{\tiny $\mathrm{M}$}}(\Gamma, \mathsf{G})$. We note that any homomorphism $\rho$ in $\mathsf{Hom}_{\hbox{\tiny $\mathrm{M}$}}(\Gamma, \mathsf{G})$ gives rise to a Margulis spacetime 
\begin{align*}
\mathsf{M}_{\rho} \defeq \rho(\Gamma)\backslash \mathbb{A} .
\end{align*}

Let us denote the space of all $\varrho$ in $\mathsf{Hom}(\Gamma, \mathsf{SO}^0(2,1))$ such that $\varrho(\Gamma)$ is Schottky by $\mathsf{Hom}_{\hbox{\tiny $\mathrm{S}$}}(\Gamma, \mathsf{SO}^0(2,1))$. We note that $\mathsf{Hom}_{\hbox{\tiny $\mathrm{S}$}}(\Gamma, \mathsf{SO}^0(2,1))$ is an analytic manifold and for any $\varrho$ in $\mathsf{Hom}_{\hbox{\tiny $\mathrm{S}$}}(\Gamma, \mathsf{SO}^0(2,1))$ the tangent space $\mathsf{T}_{\varrho}\mathsf{Hom}_{\hbox{\tiny $\mathrm{S}$}}(\Gamma, \mathsf{SO}^0(2,1))$ of $\mathsf{Hom}_{\hbox{\tiny $\mathrm{S}$}}(\Gamma, \mathsf{SO}^0(2,1))$ at the point $\varrho$ can be identified with $\mathsf{Z}^1(\varrho(\Gamma),\mathbb{R}^3)$ and $\mathsf{Hom}_{\hbox{\tiny $\mathrm{M}$}}(\Gamma, \mathsf{G})$ can be identified with a subset of the tangent bundle $\mathsf{T}\mathsf{Hom}_{\hbox{\tiny $\mathrm{S}$}}(\Gamma, \mathsf{SO}^0(2,1))$. Moreover, the following map
\begin{align}
\mathtt{L} : \mathsf{Hom}_{\hbox{\tiny $\mathrm{M}$}}(\Gamma, \mathsf{G}) &\longrightarrow\mathsf{Hom}_{\hbox{\tiny $\mathrm{S}$}}(\Gamma, \mathsf{SO}^0(2,1))\\
\notag \rho &\longmapsto \mathtt{L}_\rho
\end{align}
is surjective\cite{D2}.
\begin{lemma}
The space $\mathsf{Hom}_{\hbox{\tiny $\mathrm{M}$}}(\Gamma, \mathsf{G})$ is an analytic manifold.
\end{lemma}
\begin{proof}
We know that the space $\mathsf{Hom}_{\hbox{\tiny $\mathrm{S}$}}(\Gamma, \mathsf{SO}^0(2,1))$ is an analytic manifold. Hence the tangent bundle $\mathsf{T}\mathsf{Hom}_{\hbox{\tiny $\mathrm{S}$}}(\Gamma, \mathsf{SO}^0(2,1))$ is also an analytic manifold. Now from page number 1053 of \cite{labourie invariant} we get that the set of all $\rho$ in $\mathsf{Hom}_{\hbox{\tiny $\mathrm{M}$}}(\Gamma, \mathsf{G})$ with fixed linear part $\varrho$ is an open convex cone in $\mathsf{T}_{\varrho}\mathsf{Hom}_{\hbox{\tiny $\mathrm{S}$}}(\Gamma, \mathsf{SO}^0(2,1))$. Therefore, we conclude that $\mathsf{Hom}_{\hbox{\tiny $\mathrm{M}$}}(\Gamma, \mathsf{G})$ is an analytic manifold.
\end{proof}

Let $\rho:\Gamma\rightarrow\mathsf{G}$ be a homomomorphism such that the action of $\mathtt{L}_\rho(\Gamma)$ on $\mathbb{H}$ is Schottky. We define the $\textit{Margulis invariant}$ of an element $\gamma$ in $\Gamma$ for a given homomorphism $\rho$ as follows: if $\gamma= e$ then 
\begin{align*}
\alpha_\rho(e)\defeq0,
\end{align*}
otherwise
\begin{align*}
\alpha_\rho(\gamma) \defeq \left\langle \mathtt{u}_\rho(\gamma) \mid {\nu}_\rho\left(\gamma^-, \gamma^+\right)\right\rangle
\end{align*}
where $\mathtt{u}_\rho(\gamma)\defeq\mathtt{u}(\rho(\gamma))$ and ${\nu}_\rho\left(\gamma^-, \gamma^+\right)\defeq{\nu}\left((\mathtt{L}_\rho(\gamma))^-, (\mathtt{L}_\rho(\gamma))^+\right)$. We note that for any $\gamma$ in $\Gamma$, upto scaling ${\nu}_\rho\left(\gamma^-, \gamma^+\right)$ is the unique eigen vector of $\mathtt{L}_\rho(\gamma)$ with eigenvalue 1. Moreover, for any $\gamma$ in $\Gamma$ the element $\rho(\gamma)$ fixes a unique affine line ${l}_{\rho(\gamma)}$ in $\mathbb{A}$ and ${l}_{\rho(\gamma)}$ is parallel to the line generated by ${\nu}_\rho\left(\gamma^-, \gamma^+\right)$. Now if we consider the image of ${l}_{\rho(\gamma)}$ in $\mathsf{M}_\rho$ then it is a closed loop and its length is $\alpha_\rho(\gamma)$. (For more details see \cite{ha}, \cite{D2}, \cite{marg2}).\\

In \cite{marg1} and \cite{marg2} Margulis had shown the following result,
\begin{lemma}{[Opposite sign lemma]} 
If $\rho:\Gamma\rightarrow\mathsf{G}$ is a homomomorphism giving rise to a Margulis spacetime, then
\begin{enumerate}
\item either $\alpha_\rho(\gamma)>0$ for all $\gamma\in\Gamma$,
\item or $\alpha_\rho(\gamma)<0$ for all $\gamma\in\Gamma$.
\end{enumerate}
\end{lemma}
In \cite{labourie invariant}, Goldman--Labourie--Margulis had generalised the previous result and proved the following:
\begin{theorem}{[Goldman--Labourie--Margulis]}\label{glm}
Let $\rho:\Gamma\rightarrow\mathsf{G}$ be a homomorphism with linear part $\varrho_0$ and translation part $u$. Let $\rho$ be such that the action of $\varrho_0(\Gamma)$ on $\mathbb{H}$ is Schottky. Also let $\mathcal{C}_{\hbox{\tiny B}}(\Sigma_{\varrho_0})$ be the space of $\phi$-invariant Borel probability measures on $\mathsf{U}\Sigma_{\varrho_0}$ and $\mathcal{C}_{\hbox{\tiny per}}(\Sigma_{\varrho_0})\subset\mathcal{C}_{\hbox{\tiny B}}(\Sigma_{\varrho_0})$ be the subspace consisting of measures supported on periodic orbits. Then the following holds:
\begin{enumerate}
\item The map
\begin{align*}
\mathcal{C}_{\hbox{\tiny per}}(\Sigma_{\varrho_0}) &\longrightarrow \mathbb{R}\\
\mu_\gamma &\longmapsto \frac{\alpha_\rho(\gamma)}{\ell_{\varrho_0}(\gamma)},
\end{align*}
where $\ell_{\varrho_0}(\gamma)$ is the length of the corresponding closed geodesic of $\Sigma_{\varrho_0}$, extends to a continuous map
\begin{align*}
\mathcal{C}_{\hbox{\tiny B}}(\Sigma_{\varrho_0}) &\longrightarrow \mathbb{R}\\
\mu&\longmapsto \Upsilon_\rho(\mu).
\end{align*}
\item Moreover, the representation $\rho$ acts properly on $\mathbb{A}$ if and only if $\Upsilon_\rho(\mu) \neq 0$ for all $\mu\in\mathcal{C}_{\hbox{\tiny B}}(\Sigma_{\varrho_0})$.
\end{enumerate}
\end{theorem}
We note that the generalization of the normalized Margulis invariant as stated above was given by Labourie in \cite{l}. 

Moreover, in \cite{vari} (see also \cite{minko}) Goldman--Margulis showed:
\begin{theorem}{[Goldman--Margulis]}\label{gm}
Let $\{\varrho_t\}\subset\mathsf{Hom}_{\hbox{\tiny $\mathrm{S}$}}(\Gamma, \mathsf{SO}^0(2,1))$ be a smooth path. Also, let $\rho \in \mathsf{Hom}_{\hbox{\tiny $\mathrm{M}$}}(\Gamma, \mathsf{G})$ be such that whose linear part is $\varrho_0$ and translation part is $\dot{\varrho}_0$. Then for all $\gamma\in\Gamma$ we have
\[\left.\frac{d}{dt}\right|_{t=0}\ell_{\varrho_t}(\gamma)=\alpha_\rho(\gamma)\]
where $\ell_{\varrho_t}(\gamma)$ is the length of the closed geodesic of $\Sigma_{\varrho_t}$ corresponding to $\varrho_t(\gamma)\in\varrho_t(\Gamma)$ and $\dot{\varrho}_0\defeq\left.\frac{d}{dt}\right|_{t=0}\varrho_t$.
\end{theorem}

\subsection{Gromov geodesic flow}

Let $\Gamma$ be a non-abelian free group with finitely many generators. Let us denote the $\textit{Gromov}$ $\textit{boundary}$ of $\Gamma$ by $\partial_\infty\Gamma$ and let 
\begin{align*}
\partial_\infty\Gamma^{(2)} \defeq \partial_\infty\Gamma\times\partial_\infty\Gamma\setminus\{(x,x) \mid x\in \partial_\infty\Gamma\}.
\end{align*}
We note that there exists a natural action of $\Gamma$ on $\partial_\infty\Gamma$. The natural action of $\Gamma$ on $\partial_\infty\Gamma$ extends to a diagonal action of $\Gamma$ on $\partial_\infty\Gamma^{(2)}$. Now let
\[\widetilde{\mathsf{U}_0\Gamma} \defeq \partial_\infty\Gamma^{(2)}\times\mathbb{R}.\]
We note that $\mathbb{R}$ acts on $\widetilde{\mathsf{U}_0\Gamma}$ as translation on the last component. We denote this $\mathbb{R}$ action by $\tilde{\psi}_t$, that is,
\[\tilde{\psi}_t(x,y,s)\defeq(x,y,s+t)\]
where $x,y\in\partial_\infty\Gamma$ and $s,t\in\mathbb{R}$.
Now using Gromov's results from \cite{gromov} we get that there exists a proper cocompact action of $\Gamma$ on $\widetilde{\mathsf{U}_0\Gamma}$ which commutes with the the flow $\{\tilde{\psi}_t\}_{t\in\mathbb{R}}$ and the restriction of this action on $\partial_\infty\Gamma^{(2)}$ is the diagonal action. Moreover, from \cite{gromov} we also get that there exists a metric on $\widetilde{\mathsf{U}_0\Gamma}$ well defined up to H\"older equivalence such that the following holds:
\begin{enumerate}
\item the $\Gamma$ action is isometric,
\item every orbit of the flow $\{\tilde{\psi}_t\}_{t\in\mathbb{R}}$ gives a quasi-isometric embedding,
\item the flow $\tilde{\psi_t}$ acts by Lipschitz homeomorphisms.
\end{enumerate} 
The flow $\tilde{\psi_t}$ on $\widetilde{\mathsf{U}_0\Gamma}$ gives rise to a flow $\psi_t$ on the quotient 
\[\mathsf{U}_0\Gamma \defeq \Gamma\backslash\left(\partial_\infty\Gamma^{(2)} \times \mathbb{R}\right).\] 
We call it the $\textit{Gromov}$ $\textit{geodesic}$ $\textit{flow}$. We denote the projection onto the first coordinate of $\widetilde{\mathsf{U}_0\Gamma}$ by $\pi_1$ and the projection onto the second coordinate of $\widetilde{\mathsf{U}_0\Gamma}$ by $\pi_2$. More details about this construction can be found in Champetier \cite{champ} and Mineyev \cite{min}.

\section{Anosov representations}

In this section we will define the notion of an Anosov representation in the context of the non-semisimple Lie group $\mathsf{G} \defeq \mathsf{SO}^0(2,1)\ltimes\mathbb{R}^3$. The notion of an Anosov representation of a discrete group in a transformation group $G$ was first introduced by Labourie in \cite{orilab}. Later, Guichard--Wienhard studied Anosov representations into semisimple Lie groups in more details in \cite{guiwien}.
Recently in \cite{pressure metric}  Bridgeman--Canary--Labourie--Sambarino introduced the geodesic flow of an Anosov representation and the thermodynamical formalism in this picture, again in the context of $G$ being a semisimple group. In \cite{me}, I had studied special cases and new examples of Anosov representations when $G$ is non-semisimple. The definition given here is equivalent to the definition appearing in \cite{me}.

Let us denote the space of all affine null planes by $\mathbb{X}$ where an affine null plane is an affine plane parallel to any tangent plane of the light cone. We observe that the action of $\mathsf{G}$ on $\mathbb{X}$ is transitive. Hence for all $P\in\mathbb{X}$ we have
\[\mathbb{X} \cong \mathsf{G}/\mathsf{Stab}_{\mathsf{G}}(P).\]

\begin{definition}{[Pseudo-parabolic group]}\label{levi}
For any $P\in\mathbb{X}$ we call $\mathsf{Stab}_{\mathsf{G}}(P)$ a $\textit{pseudo-parabolic}$ subgroup of $\mathsf{G}$.
\end{definition}
Let us denote the vector space underlying a null plane $P$ by $\mathsf{V}(P)$. We recall the following proposition from subsection 4.1 of \cite{me}
\begin{proposition}\label{open}
The unique open $\mathsf{G}$ orbit for the diagonal action of $\mathsf{G}$ on the space $\mathbb{X}\times\mathbb{X}$ is, 
\[\mathcal{N}\defeq\{(P_1,P_2)\mid P_1,P_2\in\mathbb{X}, \mathsf{V}(P_1)\neq \mathsf{V}(P_2)\}.\] 
\end{proposition}

Let $\mathsf{N}$ be the space of oriented space like affine lines. We think of $\mathsf{N}$ as the space $\mathsf{U}\mathbb{A}/\sim$ where $(X,v)\sim(X_1,v_1)$ if and only if $(X_1,v_1)=\tilde{\Phi}_t(X,v)$ for some $t\in\mathbb{R}$. We denote the equivalence class of $(X,v)$ by $[(X,v)]$. We recall from subsection 4.1 of \cite{me} that
\[\mathsf{N}\cong\mathcal{N}.\]
Let $P_{X,w_1,w_2}$ be the plane passing through a point $X$ with the underlying vector space generated by the vectors $w_1$ and $w_2$. Now let us denote the vectors $(1,0,0)^t$, $(0,-1,1)^t$ and $(0,1,1)^t$ respectively by $v_0$, $v_0^-$ and $v_0^+$. We consider the following two subgroups of $\mathsf{G}$ 
\begin{align}\label{pseudo}
\mathsf{P}^\pm \defeq \mathsf{Stab}_\mathsf{G}\left(P_{O,v_0,v_0^\pm}\right)
\end{align} 
and let $\mathsf{L}=\mathsf{P}^+\cap \mathsf{P}^-$. We note that $\mathsf{L}=\mathsf{Stab}_\mathsf{G}([\mathsf{P}^+],[\mathsf{P}^-])$ for the diagonal action of $\mathsf{G}$ on $\mathsf{G}/\mathsf{P}^+\times\mathsf{G}/\mathsf{P}^-$. Moreover, using proposition \ref{open} we get that the $\mathsf{G}$ orbit of the point $([\mathsf{P}^+],[\mathsf{P}^-])\in\mathsf{G}/\mathsf{P}^+\times\mathsf{G}/\mathsf{P}^-$ is the unique open $\mathsf{G}$ orbit in $\mathsf{G}/\mathsf{P}^+\times\mathsf{G}/\mathsf{P}^-$. 

We denote the $\mathsf{G}$ orbit of the point $([\mathsf{P}^+],[\mathsf{P}^-])\in\mathsf{G}/\mathsf{P}^+\times\mathsf{G}/\mathsf{P}^-$ by $\mathsf{G}.([\mathsf{P}^+],[\mathsf{P}^-])$. If we consider the diagonal action of the group $\mathsf{G}$ on the space $\mathsf{G}.([\mathsf{P}^+],[\mathsf{P}^-])$ then the action is transitive and as $\mathsf{L}$ is the stabilizer of the point $([\mathsf{P}^+],[\mathsf{P}^-])$ we can identify
\[\mathsf{G}/\mathsf{L} \cong \mathsf{G}.([\mathsf{P}^+],[\mathsf{P}^-]).\]

Moreover, the pair $\mathsf{G}/\mathsf{P}^\pm$ gives a continuous set of foliations on the space $\mathsf{G}/\mathsf{L}$ whose tangential distributions $\mathsf{E}^\pm$ satisfy
\[\mathsf{T}(\mathsf{G}/\mathsf{L}) = \mathsf{E}^+\oplus\mathsf{E}^-.\]

We denote the Lie algebras associated to the Lie groups $\mathsf{G}, \mathsf{P}^\pm$ and $\mathsf{L}$ respectively by $\mathfrak{g}, \mathfrak{p}^\pm$ and $\mathfrak{l}$. We notice that as 
\[\dim(\mathfrak{p}^+)=\dim(\mathfrak{p}^-)=4\] 
and 
\[\dim(\mathfrak{l})=2\]
we have
\begin{align}\label{add}
\mathfrak{g} = \mathfrak{p}^+ + \mathfrak{p}^- \quad\text{and}\quad
\mathfrak{l} = \mathfrak{p}^+ \cap \mathfrak{p}^-.
\end{align}
If we complexify, we obtain the Lie algebras $\mathfrak{p}^\pm_{\mathbb{C}}$ and $\mathfrak{l}_{\mathbb{C}}$, so that the same equation \ref{add} is satisfied, that is,
\begin{align}\label{add1}
\mathfrak{g}_{\mathbb{C}} = \mathfrak{p}^+_{\mathbb{C}} + \mathfrak{p}^-_{\mathbb{C}} \quad\text{and}\quad \mathfrak{l}_{\mathbb{C}} = \mathfrak{p}^+_{\mathbb{C}} \cap \mathfrak{p}^-_{\mathbb{C}}.
\end{align}
Now as $\mathsf{SO}^0(2,1)$ is a subgroup of $\mathsf{GL}(\mathbb{R}^3)$ we get 
\[\mathsf{G}_\mathbb{C}=\mathsf{SO}(3,\mathbb{C})\ltimes\mathbb{C}^3.\]

We call a complex plane $P$ $\textit{degenerate}$ if and only if there exists a non zero vector $(v_1,v_2,v_3)^t\in P$ such that for all $(v_1^\prime ,v_2^\prime ,v_3^\prime )^t\in P$ we have
\[v_1v_1^\prime+v_2v_2^\prime+v_3v_3^\prime=0.\] Let us denote the space of all complex degenerate planes by $\mathbb{Y}_\mathbb{C}$. The group $\mathsf{SO}(3,\mathbb{C})$ acts transitively on the space $\mathbb{Y}_\mathbb{C}$. Moreover, the action of the group $\mathsf{SO}(3,\mathbb{C})$ is  transitive on the following space:
\[\mathbb{Y}_\mathbb{C}^{(2)}\defeq\{(P_1,P_2)\in\mathbb{Y}_\mathbb{C}\times\mathbb{Y}_\mathbb{C}\mid P_1\neq P_2\}.\]
Now let $\mathbb{X}_\mathbb{C}$ be the space of all affine degenerate planes in $\mathbb{C}^3$. We consider the following open subspace:
\[\mathcal{N}_\mathbb{C}\defeq\{(P_1,P_2)\in\mathbb{X}_\mathbb{C}\times\mathbb{X}_\mathbb{C}\mid \mathsf{V}(P_1)\neq \mathsf{V}(P_2)\}\]
and using the fact that $\mathsf{SO}(3,\mathbb{C})$ acts transitively on the space $\mathbb{Y}^{(2)}_\mathbb{C}$, we deduce that the action of the group $\mathsf{G}_\mathbb{C}=\mathsf{SO}(3,\mathbb{C})\ltimes\mathbb{C}^3$ on the space $\mathcal{N}_\mathbb{C}$ is transitive. Moreover, we fix $(P_1,P_2)\in\mathcal{N}_\mathbb{C}$ and observe that
\[\mathsf{L}_\mathbb{C}\cong\mathsf{Stab}_{\mathsf{G}_\mathbb{C}}(P_1,P_2)\]
where $\mathsf{L}_\mathbb{C}$ denote the complexification of the group $\mathsf{L}$. Hence 
\[\mathsf{G}_{\mathbb{C}}/\mathsf{L}_{\mathbb{C}}\cong\mathcal{N}_\mathbb{C}.\] 
Now using equation \ref{add1} we get that $\mathsf{G}_{\mathbb{C}}/\mathsf{L}_{\mathbb{C}}$ is foliated by two foliations. These foliations are respectively stabilized by $\mathsf{P}^\pm_{\mathbb{C}}$. Let us denote the tangential distributions corresponding to the foliations $\mathsf{G}_{\mathbb{C}}/\mathsf{P}^\pm_{\mathbb{C}}$ respectively by $\mathsf{E}^\pm_{\mathbb{C}}$. We observe that
\[\mathsf{T}(\mathsf{G}_{\mathbb{C}}/\mathsf{L}_{\mathbb{C}}) = \mathsf{E}^+_{\mathbb{C}}\oplus\mathsf{E}^-_{\mathbb{C}}.\]
Now let $\mathsf{E}$ be a vector bundle over a compact topological space. Also equip the total space of the bundle $\mathsf{E}$ with a flow $\{\varphi_t\}_{t\in\mathbb{R}}$ which are bundle automorphisms.
\begin{definition}
We say that the bundle $\mathsf{E}$ is $\textit{contracted}$ by the flow as $t\to\infty$ if for any metric $\|.\|$ on $\mathsf{E}$, there exists $t_0>0$, $A>0$ and $c>0$ such that for all $v\in\mathsf{E}$ and for all $t>t_0$ we have
\begin{align*}
\|\varphi_t(v)\| \leqslant Ae^{-ct}\|v\|.
\end{align*}
\end{definition}
\begin{definition}
We say that $\rho$ in $\mathsf{Hom}(\Gamma, \mathsf{G})$ (respectively $\mathsf{Hom}(\Gamma, \mathsf{G}_{\mathbb{C}}))$  is  partial $(\mathsf{G},\mathsf{P}^\pm)$-Anosov (respectively partial $(\mathsf{G}_{\mathbb{C}},\mathsf{P}^\pm_{\mathbb{C}})$-Anosov) if there exist two continuous maps
\begin{align*}
\xi^\pm_\rho : \partial_\infty\Gamma \longrightarrow \mathsf{G}/\mathsf{P}^\pm \text{ (respectively $\mathsf{G}_{\mathbb{C}}/\mathsf{P}^\pm_{\mathbb{C}}$)}
\end{align*}
such that the following conditions hold:
\begin{enumerate}
\item for all $\gamma\in\Gamma$ we have $\xi^\pm_\rho\circ\gamma = \rho(\gamma).\xi^\pm_\rho$,
\item for all $x, y\in\partial_\infty\Gamma$ if $x\neq y$ then $(\xi^+_\rho(x),\xi^-_\rho(y))\in\mathsf{G}/\mathsf{L}$ (respectively $\mathsf{G}_{\mathbb{C}}/\mathsf{L}_{\mathbb{C}}$),
\item The induced bundle $\Xi_\rho^+\defeq (\xi_\rho^+\circ\pi_1)^*\mathsf{E}^+$ (respectively $(\xi_\rho^+\circ\pi_1)^*\mathsf{E}^+_{\mathbb{C}}$) gets contracted by the lift of the flow $\tilde{\psi}_t$ as $t\to\infty$, and the induced bundle $\Xi_\rho^-\defeq (\xi_\rho^-\circ\pi_2)^*\mathsf{E}^-$ (respectively $(\xi_\rho^-\circ\pi_2)^*\mathsf{E}^-_{\mathbb{C}}$) gets contracted by the lift of the flow $\tilde{\psi}_t$ as $t\to-\infty$.
\end{enumerate}
Moreover, if the following condition holds then we say $\rho$ in $\mathsf{Hom}(\Gamma, \mathsf{G})$ is $(\mathsf{G},\mathsf{P}^\pm)$-Anosov :

\begin{enumerate}
\item[4] Let $\mathcal{B}$ over $\mathsf{U}_0\Gamma$ be the quotient bundle of the bundle
\[\{((x,y,t),z)\mid (x,y,t)\in\partial_\infty\Gamma^{(2)}\times\mathbb{R}, z\in\mathbb{R}^3\}\] over $\partial_\infty\Gamma^{(2)}\times\mathbb{R}$
under the diagonal action of $\Gamma$. There exists a H\"older section $\sigma$ of the bundle $\mathcal{B}$ over $\mathsf{U}_0\Gamma$ such that
\[\langle \nabla_\phi\sigma\mid \nu\circ(\xi_\rho^+,\xi_\rho^-)\rangle > 0\]
where $\nabla_\psi\sigma (x,y,t_0):=\frac{d}{dt}|_{t=0}\sigma(x,y,t_0+t)$.
\end{enumerate}

We call $\xi^\pm_\rho$ the $\textit{limit maps}$ associated with the partial $(\mathsf{G},\mathsf{P}^\pm)$-Anosov (respectively partial $(\mathsf{G}_{\mathbb{C}},\mathsf{P}^\pm_{\mathbb{C}})$-Anosov) representation $\rho$.
\end{definition}
 
\begin{proposition}\label{marano}
If $\rho$ is in $\mathsf{Hom}_{\hbox{\tiny $\mathrm{M}$}}(\Gamma, \mathsf{G})$ then $\rho$ is $(\mathsf{G},\mathsf{P}^\pm)$-Anosov.
\end{proposition}
\begin{proof}
Let $(X,v)\in\mathsf{U}\mathbb{A}$. Let $v^\perp$ be the plane which is perpendicular to the vector $v$ in the Lorentzian metric. We note that $v^\perp\cap\mathcal{C}$ is a disjoint union of two half lines where $\mathcal{C}$ is the upper half of $\mathsf{S}^0\backslash\{0\}$. We choose $v^\pm\in v^\perp\cap\mathcal{C}$ such that $(v^+,v,v^-)$ gives the same orientation as $(v_0^+,v_0,v_0^-)$. Let $P_{X,v,v^\pm}$ respectively be the affine null plane passing through $X$ such that its underlying vector space is generated by $v$ and $v^\pm$. We notice that $P_{X,v,v^+}\neq P_{X,v,v^-}$. Now using proposition \ref{open} we get that there exists $g_{(X,v)}\in\mathsf{G}$ such that 
\[g_{(X,v)}.P_{O,v_0,v_0^+}=P_{X,v,v^+}\]
and
\[g_{(X,v)}.P_{O,v_0,v_0^-}=P_{X,v,v^-}.\] 
Moreover, if $g_1\in\mathsf{G}$ such that
\[g_1.P_{O,v_0,v_0^+}=P_{X,v,v^+}\]
then $g_1^{-1}.g_{(X,v)}$ stabilizes the plane $P_{O,v_0,v_0^+}$. Hence $g_1^{-1}.g_{(X,v)}\in\mathsf{P}^+$. Therefore the following is a well defined map:
\begin{align*}
\eta^+ : \mathsf{U}\mathbb{A} &\longrightarrow \mathsf{G}/\mathsf{P}^+\\
(X,v) &\longmapsto [g_{(X,v)}.\mathsf{P}^+].
\end{align*}
We notice that $\eta^+$ is $\mathsf{G}$-equivariant. Similarly, we define another $\mathsf{G}$-equivariant map
\begin{align*}
\eta^- : \mathsf{U}\mathbb{A} &\longrightarrow \mathsf{G}/\mathsf{P}^-\\
(X,v) &\longmapsto [g_{(X,v)}.\mathsf{P}^-].
\end{align*}
Also for all $(X,v)\in \mathsf{U}\mathbb{A}$ we see that
\[(\eta^+,\eta^-)(X,v)=([g_{(X,v)}.\mathsf{P}^+],[g_{(X,v)}.\mathsf{P}^-])=g_{(X,v)}.([\mathsf{P}^+],[\mathsf{P}^-]).\]
Hence $(\eta^+,\eta^-)(\mathsf{U}\mathbb{A})\subset\mathsf{G}/\mathsf{L}$.

Let $\rho\in\mathsf{Hom}_{\hbox{\tiny $\mathrm{M}$}}(\Gamma, \mathsf{G})$. Hence $\mathtt{L}_\rho\in\mathsf{Hom}_{\hbox{\tiny $\mathrm{S}$}}(\Gamma, \mathsf{SO}^0(2,1))$. Now $\Gamma$ being a free group we get that there exists a $\Gamma$-equivariant homeomorphism 
\[\iota_\rho : \partial_\infty\Gamma \longrightarrow\Lambda_\infty\mathtt{L}_\rho(\Gamma).\] 
We define 
\[\eta_\rho^\pm\defeq\left.\eta^\pm\right|_{\mathsf{U}^\rho_{\hbox{\tiny $\mathrm{rec}$}}\mathbb{A}} \]
and observe that for any $[g.\mathsf{P}^+]\in\mathsf{G}/\mathsf{P}^+$ we have
\begin{align*}
(\eta_\rho^+)^{-1}\left([g.\mathsf{P}^+]\right)= \{g.O + t\mathtt{L}(g)v_0 &+ s_1 \mathtt{L}(g)v_0^+ , \mathtt{L}(g)v_0\\ 
&+ s_2 \mathtt{L}(g)v_0^+)\mid t, s_1, s_2 \in \mathbb{R}\}\cap\mathsf{U}^\rho_{\hbox{\tiny $\mathrm{rec}$}}\mathbb{A}.
\end{align*}
Now using proposition 3.2.6 of \cite{me} we notice that the maps $\eta^\pm_\rho\circ\mathtt{N}_\rho$ gives rise to a pair of $\Gamma$-equivariant continuous maps
\[\zeta^\pm_\rho : \Lambda_\infty\mathtt{L}_\rho(\Gamma) \longrightarrow \mathsf{G}/\mathsf{P}^\pm.\]
Therefore the following map,
\[\xi^\pm_\rho \defeq \zeta^\pm_\rho\circ\iota_\rho: \partial_\infty\Gamma \longrightarrow \mathsf{G}/\mathsf{P}^\pm\] is also continuous and $\Gamma$-equivariant. Moreover, as $(\eta_\rho^+,\eta_\rho^-)(\mathsf{U}^\rho_{\hbox{\tiny $\mathrm{rec}$}}\mathbb{A})\subset\mathsf{G}/\mathsf{L}$ we get that if $x,y\in\partial_\infty\Gamma$ with $x\neq y$ then $(\xi^+_\rho(x),\xi^-_\rho(y))\in\mathsf{G}/\mathsf{L}.$
We also observe that 
\[\mathsf{T}_{[g.\mathsf{P}^\pm]}\mathsf{G}/\mathsf{P}^\pm \cong \mathbb{R}(0,\mathtt{L}(g)v_0^\mp)\oplus\mathbb{R}(\mathtt{L}(g)v_0^\mp,0).\]
Now using proposition 3.3.1 of \cite{me} and and Theorem \ref{N} we conclude that $\rho$ is $(\mathsf{G},\mathsf{P}^\pm)$-Anosov .
\end{proof}

\section{Deformation theory}\label{deform}

\subsection{Analyticity of limit maps}

In this section we will show that the limit maps vary analytically over the analytic manifold $\mathsf{Hom}_{\hbox{\tiny $\mathrm{M}$}}(\Gamma, \mathsf{G})$. The proofs given in this section are similar to some of the proofs given in the section 6 of \cite{pressure metric} the only difference being that in our case the group $\mathsf{G}$ is not semi-simple.

\begin{theorem} \label{lanal}
Let $\{\rho_u\}_{u\in \mathcal{D}}\subset\mathsf{Hom}(\Gamma, \mathsf{G})$ be a real analytic family of injective homomorphisms parameterized by an open disk $\mathcal{D}$ around 0. Also let $\rho_0$ be a partial $(\mathsf{G},\mathsf{P}^\pm)$-Anosov representation (where $\mathsf{G}$ is the non-semisimple Lie group $\mathsf{SO}^0(2,1)\ltimes\mathbb{R}^3$ and $\mathsf{P}^\pm$ are pseudo-parabolic subgroups of $\mathsf{G}$ as mentioned in equation \ref{pseudo})  with limit maps 
\[\xi^\pm_0 : \partial_{\infty}\Gamma \longrightarrow \mathsf{G}/\mathsf{P}^\pm.\] 
Then there exists an open disk $\mathcal{D}_0$ containing 0, $\mathcal{D}_0\subset\mathcal{D}$, a positive real number $\mu$ and a pair of continuous maps
\[\xi^\pm : \mathcal{D}_0\times\partial_{\infty}\Gamma \longrightarrow \mathsf{G}/\mathsf{P}^\pm\] such that the following conditions hold:
\begin{enumerate}
\item for all $x\in\partial_{\infty}\Gamma$ we have $\xi^\pm(0,x)=\xi^\pm_0(x)$,
\item the representation $\rho_u$ is partial $(\mathsf{G},\mathsf{P}^\pm)$-Anosov for all $u\in\mathcal{D}_0$  and the corresponding limit maps denoted by
\begin{align*}
\xi^\pm_u : \partial_{\infty}\Gamma &\longrightarrow \mathsf{G}/\mathsf{P}^\pm\\
x &\longmapsto \xi^\pm(u,x),
\end{align*}
are $\mu$-H\"older continuous,
\item for all $x\in\partial_{\infty}\Gamma$ the following maps are real analytic
\begin{align*}
\xi^\pm_x : \mathcal{D}_0 &\longrightarrow \mathsf{G}/\mathsf{P}^\pm\\
u &\longmapsto \xi^\pm(u, x),
\end{align*}
\item the following maps are $\mu$-H\"older continuous
\begin{align*}
\xi_\dagger^\pm:\partial_{\infty}\Gamma&\longrightarrow\mathcal{C}^\omega (\mathcal{D}_0, \mathsf{G}/\mathsf{P}^\pm)\\
x &\longmapsto \xi^\pm_x
\end{align*} 
where $\mathcal{C}^\omega (\mathcal{D}_0, \mathsf{G}/\mathsf{P}^\pm)$ is the space of all real analytic functions from $\mathcal{D}_0$ to $\mathsf{G}/\mathsf{P}^\pm$.
\item the following maps are real analytic
\begin{align*}
\xi_\ddagger^\pm:\mathcal{D}_0 &\longrightarrow\mathcal{C}^\mu(\partial_{\infty}\Gamma, \mathsf{G}/\mathsf{P}^\pm)\\
u &\longmapsto \xi^\pm_u
\end{align*} 
where $\mathcal{C}^\mu(\partial_{\infty}\Gamma, \mathsf{G}/\mathsf{P}^\pm)$ is the space of all $\mu$-H\"older continuous maps from $\partial_{\infty}\Gamma$ to $\mathsf{G}/\mathsf{P}^\pm$.
\end{enumerate}
\end{theorem}

We note that Theorem \ref{lanal} will be proved using the following more general result.

\begin{theorem}\label{canal} 
Let $\{\rho_u\}_{u\in \mathcal{D}^{\mathbb{C}}}\subset\mathsf{Hom}(\Gamma, \mathsf{G}_{\mathbb{C}})$ be a complex analytic family of injective homomorphisms parameterized by an open disk $\mathcal{D}^{\mathbb{C}}$ around 0. Also let $\rho_0$ be a partial $(\mathsf{G}_{\mathbb{C}},\mathsf{P}^\pm_{\mathbb{C}})$-Anosov representation (where $\mathsf{G}_{\mathbb{C}}\cong\mathsf{SO}(3,\mathbb{C})\ltimes\mathbb{C}^3$)  with limit maps 
\[\xi^\pm_0 : \partial_{\infty}\Gamma \longrightarrow \mathsf{G}_{\mathbb{C}}/\mathsf{P}^\pm_{\mathbb{C}}.\] 
Then there exists an open disk $\mathcal{D}_0^{\mathbb{C}}$ containing 0, $\mathcal{D}^{\mathbb{C}}_0\subset\mathcal{D}^{\mathbb{C}}$, a positive real number $\mu$ and a pair of continuous maps
\[\xi^\pm : \mathcal{D}_0^{\mathbb{C}}\times\partial_{\infty}\Gamma \longrightarrow \mathsf{G}_{\mathbb{C}}/\mathsf{P}^\pm_{\mathbb{C}}\] such that the following conditions hold:
\begin{enumerate}
\item for all $x\in\partial_{\infty}\Gamma$ we have $\xi^\pm(0,x)=\xi^\pm_0(x)$,
\item the representation $\rho_u$ is partial $(\mathsf{G}_{\mathbb{C}},\mathsf{P}^\pm_{\mathbb{C}})$-Anosov for all $u\in\mathcal{D}_0^{\mathbb{C}}$  and the corresponding limit maps denoted by
\begin{align*}
\xi^\pm_u : \partial_{\infty}\Gamma &\longrightarrow \mathsf{G}_{\mathbb{C}}/\mathsf{P}^\pm_{\mathbb{C}}\\
x &\longmapsto \xi^\pm(u,x),
\end{align*}
are $\mu$-H\"older continuous,
\item for all $x\in\partial_{\infty}\Gamma$ the following maps are real analytic
\begin{align*}
\xi^\pm_x : \mathcal{D}_0^{\mathbb{C}} &\longrightarrow \mathsf{G}_{\mathbb{C}}/\mathsf{P}^\pm_{\mathbb{C}}\\
u &\longmapsto \xi^\pm(u, x),
\end{align*}
\item the following maps are $\mu$-H\"older continuous
\begin{align*}
\xi_\dagger^\pm:\partial_{\infty}\Gamma&\longrightarrow\mathcal{C}^\omega (\mathcal{D}_0^{\mathbb{C}}, \mathsf{G}_{\mathbb{C}}/\mathsf{P}^\pm_{\mathbb{C}})\\
x &\longmapsto \xi^\pm_x
\end{align*} 
where $\mathcal{C}^\omega (\mathcal{D}_0^{\mathbb{C}}, \mathsf{G}_{\mathbb{C}}/\mathsf{P}^\pm_{\mathbb{C}})$ is the space of all complex analytic functions from $\mathcal{D}_0^{\mathbb{C}}$ to $\mathsf{G}_{\mathbb{C}}/\mathsf{P}^\pm_{\mathbb{C}}$.
\item the following maps are complex analytic
\begin{align*}
\xi_\ddagger^\pm:\mathcal{D}_0^{\mathbb{C}} &\longrightarrow\mathcal{C}^\mu(\partial_{\infty}\Gamma, \mathsf{G}_{\mathbb{C}}/\mathsf{P}^\pm_{\mathbb{C}})\\
u &\longmapsto \xi^\pm_u
\end{align*} 
where $\mathcal{C}^\mu(\partial_{\infty}\Gamma, \mathsf{G}_{\mathbb{C}}/\mathsf{P}^\pm_{\mathbb{C}})$ is the space of all $\mu$-H\"older continuous maps from $\partial_{\infty}\Gamma$ to $\mathsf{G}_{\mathbb{C}}/\mathsf{P}^\pm_{\mathbb{C}}$.
\end{enumerate}
\end{theorem}

\begin{proof}
Let $\{\rho_u\}_{u\in\mathcal{D}^{\mathbb{C}}}\subset\mathsf{Hom}(\Gamma,\mathsf{G}_{\mathbb{C}})$ be a complex analytic family of injective homomorphisms such that $\rho_0$ is partial $(\mathsf{G}_{\mathbb{C}},\mathsf{P}^\pm_{\mathbb{C}})$-Anosov. Now we consider the trivial $\mathsf{G}_{\mathbb{C}}/\mathsf{P}^+_{\mathbb{C}}$-bundle over $\mathcal{D}^{\mathbb{C}}\times\widetilde{\mathsf{U}_0\Gamma}$ as follows:
\begin{align*}
\pi : \tilde{A} \defeq  \mathcal{D}^{\mathbb{C}}\times\widetilde{\mathsf{U}_0\Gamma}\times\mathsf{G}_{\mathbb{C}}/\mathsf{P}^+_{\mathbb{C}} \longrightarrow \mathcal{D}^{\mathbb{C}}\times\widetilde{\mathsf{U}_0\Gamma}.
\end{align*}
Furthermore, we consider the following action of $\Gamma$ on $\tilde{A}$
\[\gamma (u, x, [g]) = (u, \gamma(x), [\rho_u (\gamma)g])\]
where $\gamma\in\Gamma$. We notice that the quotient bundle $A \defeq \Gamma\backslash\tilde{A}$ is a Lipschitz transversely complex analytic $\mathsf{G}_{\mathbb{C}}/\mathsf{P}^+_{\mathbb{C}}$-bundle over $\mathcal{D}^{\mathbb{C}}\times\mathsf{U}_0\Gamma$ (see definition 6.3 and definition 6.4 of \cite{pressure metric} for a definition of transverse analyticity). 

Let us denote the geodesic flow on $\tilde{A}$ which is a lift of the geodesic flow $\{\tilde{\psi}_t\}_{t\in\mathbb{R}}$ on $\widetilde{\mathsf{U}_0\Gamma}$ by $\{\tilde{\Psi}_t\}_{t\in\mathbb{R}}$. Moreover, we choose a geodesic flow $\{\Psi_t\}_{t\in\mathbb{R}}$ on $A$ which is a lift of the geodesic flow $\{\psi_t\}_{t\in\mathbb{R}}$ on $\mathsf{U}_0\Gamma$. We note that the flow  $\{\tilde{\Psi}_t\}_{t\in\mathbb{R}}$ acts trivially on the $\mathsf{G}_{\mathbb{C}}/\mathsf{P}^+_{\mathbb{C}}$ and $\mathcal{D}^{\mathbb{C}}$ components.\\
Now as $\rho_0$ is partial $(\mathsf{G}_{\mathbb{C}},\mathsf{P}^\pm_{\mathbb{C}})$-Anosov with limit maps
\[\xi^\pm_0: \partial_{\infty}\Gamma \rightarrow \mathsf{G}_{\mathbb{C}}/\mathsf{P}^\pm_{\mathbb{C}},\]
the following map $\tilde{\sigma}_0$ defines a $\Gamma$-equivariant section of the restriction of the bundle $\tilde{A}$ over $\{0\}\times\widetilde{\mathsf{U}_0\Gamma}$,
\begin{align*}
\tilde{\sigma}_0 : \{0\}\times\widetilde{\mathsf{U}_0\Gamma} &\longrightarrow \tilde{A}\\
(0, (x, y, t)) &\longmapsto (0, (x, y, t), \xi^+_0(x)).
\end{align*}
Therefore the section $\tilde{\sigma}_0$ gives rise to a section $\sigma_0$ of $A$ over $\{0\}\times\mathsf{U}_0\Gamma$.

Since $\rho_0$ is partial $(\mathsf{G}_{\mathbb{C}},\mathsf{P}^\pm_{\mathbb{C}})$-Anosov, the bundle $\Xi^+_{\rho_0}$ over $\{0\}\times\mathsf{U}_0\Gamma$ with fiber $\mathsf{T}_{\sigma_0(0, \mathfrak{X})}\pi^{-1}(0, \mathfrak{X})$ gets contracted by the lift of the geodesic flow $\psi_t$ as $t$ goes to $\infty$. Hence there exists a real number $t_0$ such that for all $\mathfrak{X}$ in $\mathsf{U}_0\Gamma$ we have
\[\left\Vert\left(\mathsf{D}^{\psi_{t_0}}\Psi_{t_0}\right)_{\sigma_0(0,\mathfrak{X})}\right\Vert< 1\]
where
\[\left(\mathsf{D}^{\psi_{t_0}}\Psi_{t_0}\right)_{\sigma_0(0,\mathfrak{X})}: \mathsf{T}_{\sigma_0(0, \mathfrak{X})}\pi^{-1}(0, \mathfrak{X})\rightarrow \mathsf{T}_{\sigma_0(0, \psi_{t_0}\mathfrak{X})}\pi^{-1}(0, \psi_{t_0}\mathfrak{X})\]
is the fiberwise map of the bundle automorphism induced by $\psi_{t_0}$ or in short ``lift of $\psi_{t_0}$". Now using theorem 6.5 of \cite{pressure metric} we get that there exists an open disk $\mathcal{D}^{\mathbb{C}}_1$ containing $0$, $\mathcal{D}^{\mathbb{C}}_1\subset \mathcal{D}^{\mathbb{C}}$, a positive real number $\mu$, and a $\mu$-H\"older transversely complex analytic section (see definition 6.3 and definition 6.4 of \cite{pressure metric} for a definition of transverse analyticity)
\[\sigma : \mathcal{D}^{\mathbb{C}}_1\times\mathsf{U}_0\Gamma \longrightarrow A\] such that the following holds:
\begin{enumerate}
\item $\left.\sigma\right|_{\{0\}\times\mathsf{U}_0\Gamma}=\sigma_0$,
\item $\Psi_{t_0}(\sigma(u,.))=\sigma(u,\psi_{t_0}(.))$,
\item for all $\mathfrak{X}\in\mathsf{U}_0\Gamma$ and $u\in\mathcal{D}^{\mathbb{C}}_1$ we have
 \[\left\Vert\left(\mathsf{D}^{\psi_{t_0}}\Psi_{t_0}\right)_{\sigma(u,\mathfrak{X})}\right\Vert< 1.\]
\end{enumerate}
Now using theorem 6.5 (4) of \cite{pressure metric} we deduce that for all real number $t$
\[\Psi_t(\sigma(u,.))=\sigma(u,\psi_t(.)).\]
Therefore we get that there exists an open disk $\mathcal{D}^{\mathbb{C}}_1$ containing $0$, $\mathcal{D}^{\mathbb{C}}_1\subset \mathcal{D}^{\mathbb{C}}$, a positive real number $\mu$, and a $\mu$-H\"older transversely complex analytic section $\sigma$ of the bundle $A$ such that
\begin{enumerate}
\item $\left.\sigma\right|_{\{0\}\times\mathsf{U}_0\Gamma}=\sigma_0$,
\item for all $t\in\mathbb{R}$ we have $\Psi_t(\sigma(u,.))=\sigma(u,\psi_t(.))$,
\item $\Psi_t$ is contracting along $\sigma$ as $t$ goes to $\infty$.
\end{enumerate}
Now we can lift the section $\sigma$ to get a section $\tilde{\sigma}$ as follows:
\[\tilde{\sigma} : \mathcal{D}^{\mathbb{C}}_1\times\widetilde{\mathsf{U}_0\Gamma} \rightarrow \tilde{A} = \mathcal{D}^{\mathbb{C}}_1\times\widetilde{\mathsf{U}_0\Gamma}\times\mathsf{G}_{\mathbb{C}}/\mathsf{P}^+_{\mathbb{C}}.\]
Let $\pi_3$ be the projection of $\mathcal{D}^{\mathbb{C}}_1\times\widetilde{\mathsf{U}_0\Gamma}\times\mathsf{G}_{\mathbb{C}}/\mathsf{P}^+_{\mathbb{C}}$ onto $\mathsf{G}_{\mathbb{C}}/\mathsf{P}^+_{\mathbb{C}}$. Therefore we get a map
\[\eta\defeq\pi_3\circ\tilde{\sigma} : \mathcal{D}^{\mathbb{C}}_1\times\widetilde{\mathsf{U}_0\Gamma} \rightarrow \mathsf{G}_{\mathbb{C}}/\mathsf{P}^+_{\mathbb{C}}.\]
Since $\Psi_t(\sigma(u,.))=\sigma(u,\psi_t(.))$ for all $t\in\mathbb{R}$ we get that the map $\eta$ is invariant under the flow $\{\tilde{\psi}_t\}_{t\in\mathbb{R}}$. Hence the expression $\eta(u, (x, y, t))$ is independent of the variable $t$.

Now let $\gamma\in\Gamma$ be an element of infinite order with period $t_\gamma$ i.e. 
\[\gamma(\gamma^-,\gamma^+,0)=(\gamma^-,\gamma^+, t_\gamma).\]
We notice that as $\eta_u(\gamma^-,\gamma^+,t)$ is independent of the variable $t$, we have
\[\gamma^{-n}\eta_u(\gamma^-,\gamma^+,0)=\eta_u(\gamma^-,\gamma^+,-nt_\gamma)= \eta_u(\gamma^-,\gamma^+,0)\]
and hence $\eta_u(\gamma^-,\gamma^+,0)$ is a fixed point of $\gamma^{-1}$. We claim that it is an attracting fixed point. Indeed, as $\tilde{\Psi}_t$ is contracting as $t$ goes to $\infty$ and as $\|.\|$ is $\Gamma$-equivariant we have for all $X$ in $\mathsf{T}_{\eta_u(\gamma^-,\gamma^+,0)}\mathsf{G}_\mathbb{C}/\mathsf{P}^+_\mathbb{C}$ that
\begin{align*}
\|\gamma^{-n}X\|_{\eta_u(\gamma^-,\gamma^+,0)} &= \|X\|_{\eta_u(\gamma^n(\gamma^-,\gamma^+,0))}\\
&= \|X\|_{\eta_u(\gamma^-,\gamma^+,nt_\gamma)}\\
&= \|X\|_{\tilde{\Psi}_{nt_\gamma}\eta_u(\gamma^-,\gamma^+,0)} \leqslant Ae^{-ct_\gamma n}\|X\|_{\eta_u(\gamma^-,\gamma^+,0)}.
\end{align*}
Hence for $m$ large enough the operator norm $\|\gamma^{-m}\|<1$ and we have that there exists a ball $\mathsf{B}_{d}(\eta_u(\gamma^-,\gamma^+,0),k_0)$ of radius $k_0$ around $\eta_u(\gamma^-,\gamma^+,0)$ for some metric $d$ on $\mathsf{G}_\mathbb{C}/\mathsf{P}^+_\mathbb{C}$ such that $\gamma^{-m}$ is contracting on the ball. Hence $\gamma^{-1}$ is also contracting on the ball. We call the ball $\mathsf{B}_{d}(\eta_u(\gamma^-,\gamma^+,0),k_0)$ a $\textit{basin of convergence}$ for the action of $\gamma^{-1}$ around $\eta_u(\gamma^-,\gamma^+,0)$. Therefore in particular for any sequence $\{p_n\}_{n\in\mathbb{N}}$ in $\mathsf{B}_{d}(\eta_u(\gamma^-,\gamma^+,0),k_0)$ we have that
\[\lim_{n\to\infty}d(\eta_u(\gamma^-,\gamma^+,0),\gamma^{-n}p_n)=0.\]
Moreover, for any $\Gamma$-invariant metric $\mathfrak{d}$ on $\widetilde{\mathsf{U}_0\Gamma}$ and given any $z\in\partial_\infty\Gamma$ there exists $t_z$ such that 
\[\lim_{t\to-\infty}\mathfrak{d}(\tilde{\psi}_t(\gamma^-,\gamma^+,0),\tilde{\psi}_t(\gamma^-,z,t_z))=0.\]
Hence
\begin{align*}
0&=\lim_{n\to\infty}\mathfrak{d}((\gamma^-,\gamma^+,-nt_\gamma),(\gamma^-,z,t_z-nt_\gamma))\\
&=\lim_{n\to\infty}\mathfrak{d}(\gamma^{-n}(\gamma^-,\gamma^+,0),(\gamma^-,z,t_z-nt_\gamma))\\
&=\lim_{n\to\infty}\mathfrak{d}((\gamma^-,\gamma^+,0),\gamma^n(\gamma^-,z,t_z-nt_\gamma)).
\end{align*}
Therefore if we take $p_n=\eta_u(\gamma^n(\gamma^-,z,t_z-nt_\gamma))$ then the sequence is eventually in $\mathsf{B}_d(\eta_u(\gamma^-,\gamma^+,0),k_0)$ and we get that
\begin{align*}
0&=\lim_{n\to\infty}d(\eta_u(\gamma^-,\gamma^+,0),\gamma^{-n}\eta_u(\gamma^n(\gamma^-,z,t_z-nt_\gamma)))\\
&=\lim_{n\to\infty}d(\eta_u(\gamma^-,\gamma^+,0),\eta_u(\gamma^-,z,t_z-nt_\gamma)).
\end{align*}
Now as $\eta(u, (x, y, t))$ is independent of $t$ we get that
\[0=\lim_{n\to\infty}d(\eta_u(\gamma^-,\gamma^+,0),\eta_u(\gamma^-,z,0))\]
and hence $\eta_u(\gamma^-,\gamma^+,0)=\eta_u(\gamma^-,z,0)$.
Moreover, as the fixed points of infinite order elements are dense in $\partial_\infty\Gamma$ we conclude that $\eta(u, (x, y, t))$ is independent of the variable $y$. Therefore there exists a $\Gamma$-equivariant H\"older transversely complex analytic map
\[\xi^+: \mathcal{D}^{\mathbb{C}}_1\times\partial_\infty\Gamma \longrightarrow \mathsf{G}_{\mathbb{C}}/\mathsf{P}^+_{\mathbb{C}}\]
such that for all $x\in\partial_{\infty}\Gamma$ we have $\xi^+(0,x)=\xi^+_0(x)$.

In a similar way we get that there exists an open disk $\mathcal{D}^{\mathbb{C}}_2$ containing $0$, $\mathcal{D}^{\mathbb{C}}_2\subset \mathcal{D}^{\mathbb{C}}$ such that there exists a $\Gamma$-equivariant H\"older transversely complex analytic map
\[\xi^-: \mathcal{D}^{\mathbb{C}}_2\times\partial_\infty\Gamma \longrightarrow \mathsf{G}_{\mathbb{C}}/\mathsf{P}^-_{\mathbb{C}}\]
such that for all $x\in\partial_{\infty}\Gamma$ we have $\xi^-(0,x)=\xi^-_0(x)$. \\
Moreover, we recall that $\mathcal{N}_\mathbb{C}$ is open in $\mathbb{X}_\mathbb{C}\times\mathbb{X}_\mathbb{C}$ and we know that
\[\mathsf{G}_{\mathbb{C}}/\mathsf{L}_{\mathbb{C}}\cong\mathcal{N}_\mathbb{C}.\]
Hence $\mathsf{G}_{\mathbb{C}}/\mathsf{L}_{\mathbb{C}}$ is an open subset of $\mathsf{G}_{\mathbb{C}}/\mathsf{P}^+_{\mathbb{C}}\times\mathsf{G}_{\mathbb{C}}/\mathsf{P}^-_{\mathbb{C}}$. Now as 
\[(\xi^+_0,\xi^-_0)(\partial_\infty\Gamma^{(2)})\subset\mathsf{G}_{\mathbb{C}}/\mathsf{L}_{\mathbb{C}},\] we get that there exists an open disk $\mathcal{D}^{\mathbb{C}}_0$ containing $0$, $\mathcal{D}^{\mathbb{C}}_0\subset \mathcal{D}^{\mathbb{C}}_1\cap\mathcal{D}^{\mathbb{C}}_2$ such that 
\[(\xi^+,\xi^-)(\mathcal{D}^{\mathbb{C}}_0\times\partial_\infty\Gamma^{(2)})\subset\mathsf{G}_{\mathbb{C}}/\mathsf{L}_{\mathbb{C}}.\] 
Therefore we have proved properties (1), (2), (3) and (4). Now using lemma 6.8 of \cite{pressure metric} we get property (5).
\end{proof}

\begin{proof}[Proof of Theorem \ref{lanal}] 
Let $\{\rho_u\}_{u\in\mathcal{D}}\subset\mathsf{Hom}(\Gamma,\mathsf{G})$ be a real analytic family of injective homomorphisms parametrized by a real open disk $\mathcal{D}$ such that $\rho_0$ is partial $(\mathsf{G},\mathsf{P}^\pm)$-Anosov. We note that if a representation $\rho$ is partial $(\mathsf{G},\mathsf{P}^\pm)$-Anosov then it is also partial $(\mathsf{G}_{\mathbb{C}},\mathsf{P}^\pm_{\mathbb{C}})$-Anosov. Hence $\rho_0$ is partial $(\mathsf{G}_{\mathbb{C}},\mathsf{P}^\pm_{\mathbb{C}})$-Anosov. Furthermore, we can choose a real open disk $\mathcal{D}_3$ containing 0, $\mathcal{D}_3\subset\mathcal{D}$ such that its complexification $\mathcal{D}^{\mathbb{C}}_3$ parametrizes a family of representations $\{\rho_u\}_{u\in\mathcal{D}^{\mathbb{C}}_3}\subset\mathsf{Hom}(\Gamma,\mathsf{G}_{\mathbb{C}})$.

Now using theorem \ref{canal} we get that there exists an open disc $\mathcal{D}^{\mathbb{C}}_{00}$ containing 0, $\mathcal{D}^{\mathbb{C}}_{00}\subset\mathcal{D}^{\mathbb{C}}_3$ and a $\Gamma$-equivariant H\"older transversely complex analytic map 
\[\xi^+:\mathcal{D}^{\mathbb{C}}_{00}\times\partial_{\infty}\Gamma \longrightarrow \mathsf{G}_{\mathbb{C}}/\mathsf{P}^+_{\mathbb{C}}\] such that for all $x\in\partial_{\infty}\Gamma$ we have
\[\xi^+(0,x)=\xi^+_0(x).\]
We claim that there exists an open disk $\mathcal{D}_{01}\subset\mathcal{D}_{00}$, containing $0$, such that 
\[\xi^+(\mathcal{D}_{01}\times\partial_{\infty}\Gamma)\subset\mathsf{G}/\mathsf{P}^+.\]
Indeed, to begin with we notice that $\xi^+(\{0\}\times\partial_{\infty}\Gamma)\subset\mathsf{G}/\mathsf{P}^+$. Now using theorem 6.5 (4) of \cite{pressure metric} we get that there exists an open disk $\mathcal{D}_4^{\mathbb{C}}\subset\mathcal{D}_{00}^{\mathbb{C}}$, containing $0$, and a neighborhood $\mathsf{B}$ of $\xi^+(\mathcal{D}_4^{\mathbb{C}}\times\partial_{\infty}\Gamma)$ such that the limit map is unique in $\mathsf{B}$. Let $\mathfrak{i}$ be the anti-holomorphic involution on $\mathsf{G}_{\mathbb{C}}/\mathsf{P}^+_{\mathbb{C}}$. As $\mathfrak{i}$ is continuous and $\mathfrak{i}\circ\xi^+_0=\xi^+_0$ we obtain that there exists an open disk $\mathcal{D}_5^{\mathbb{C}}\subset\mathcal{D}_{00}^{\mathbb{C}}$, containing $0$, such that
\[\mathfrak{i}\circ\xi^+(\mathcal{D}_5^{\mathbb{C}}\times\partial_{\infty}\Gamma)\subset\mathsf{B}.\]
We define 
\[\mathcal{D}_{01}^{\mathbb{C}}\defeq\mathcal{D}_4^{\mathbb{C}}\cap\mathcal{D}_5^{\mathbb{C}}\]
and by local uniqueness of the limit map we notice that for all $u\in\mathcal{D}_{01}^{\mathbb{C}}$ the following holds:
\[\mathfrak{i}\circ\xi^+_u=\xi^+_{iu}.\]
Now for all $u\in\mathcal{D}_{01}^{\mathbb{C}}$ satisfying $\mathfrak{i}\circ\rho_u=\rho_u$ we get that $u=iu$ and hence we conclude that
\[\mathfrak{i}\circ\xi^+_u=\xi^+_u.\]
Similarly we can choose an open disk $\mathcal{D}_{02}\subset\mathcal{D}_{00}$, containing $0$, such that 
\[\xi^-(\mathcal{D}_{02}\times\partial_{\infty}\Gamma)\subset\mathsf{G}/\mathsf{P}^-.\]
Now using the fact that $\mathsf{G}/\mathsf{L}$ is an open subset of $\mathsf{G}/\mathsf{P}^-\times\mathsf{G}/\mathsf{P}^+$ and the fact that the restrictions of complex analytic functions to real analytic submanifolds are real analytic we get that we can choose an open disk $\mathcal{D}_0\subset\mathcal{D}_{01}\cap\mathcal{D}_{02}$ containing 0 such that the maps $\left.\xi^\pm\right|_{\mathcal{D}_0}$ satisfies all the properties required by Theorem \ref{lanal}.
\end{proof}

\subsection{Analyticity of reparametrizations}

Let $\mathsf{U}_0\Gamma$ be the Gromov geodesic flow of the free group $\Gamma$ and let $\rho$ be an element of $\mathsf{Hom}_{\hbox{\tiny $\mathrm{M}$}}(\Gamma, \mathsf{G})$. Moreover, let $\Sigma_{\mathtt{L}(\rho)}\defeq\mathtt{L}_\rho(\Gamma)\backslash\mathbb{H}$ and $\mathsf{M}_\rho\defeq\rho(\Gamma)\backslash\mathbb{A}$. Now as $\Gamma$ is a free group we have an orbit equivalent homeomorphism between $\mathsf{U}_0\Gamma$ and $\mathsf{U}_{\hbox{\tiny $\mathrm{rec}$}}\Sigma_{\mathtt{L}(\rho)}$. Moreover, the flow on $\mathsf{U}_{\hbox{\tiny $\mathrm{rec}$}}\Sigma_{\mathtt{L}(\rho)}$ coming from the geodesic flow on $\mathsf{U}\Sigma_{\mathtt{L}(\rho)}$ is a H\"older reparametrization of the Gromov flow on $\mathsf{U}_0\Gamma$. Also from \cite{labourie invariant} and \cite{geodesic} we know that there exists an orbit equivalent homeomorphism between $\mathsf{U}_{\hbox{\tiny $\mathrm{rec}$}}\Sigma_{\mathtt{L}(\rho)}$ and $\mathsf{U}_{\hbox{\tiny $\mathrm{rec}$}}\mathsf{M}_\rho$ such that the flow on $\mathsf{U}_{\hbox{\tiny $\mathrm{rec}$}}\mathsf{M}_\rho$ coming from the affine linear flow is a H\"older reparametrization of the flow on $\mathsf{U}_{\hbox{\tiny $\mathrm{rec}$}}\Sigma_{\mathtt{L}(\rho)}$ coming from the geodesic flow on $\mathsf{U}\Sigma_{\mathtt{L}(\rho)}$. Therefore there exists an orbit equivalent homeomorphism between $\mathsf{U}_0\Gamma$ and $\mathsf{U}_{\hbox{\tiny $\mathrm{rec}$}}\mathsf{M}_\rho$ such that the affine linear flow on $\mathsf{U}_{\hbox{\tiny $\mathrm{rec}$}}\mathsf{M}_\rho$ is a H\"older reparametrization of the Gromov flow.  Hence for any $\rho\in\mathsf{Hom}_{\hbox{\tiny $\mathrm{M}$}}(\Gamma, \mathsf{G})$ we get a positive H\"older continuous map
\[\mathfrak{f}_\rho : \mathsf{U}_0\Gamma \rightarrow \mathbb{R}\] 
which gives the reparametrization. We recall that positivity follows from lemma 3 of \cite{geodesic}. We further note that for all $\gamma\in\Gamma$ we have
\[\int_\gamma\mathfrak{f}_\rho=\alpha(\rho)(\gamma).\]

\begin{proposition}\label{ranal}
Let $\{\rho_u\}_{u\in \mathcal{D}}\subset\mathsf{Hom}_{\hbox{\tiny $\mathrm{M}$}}(\Gamma, \mathsf{G})$ be a real analytic family of injective homomorphisms parametrized by an open disk $\mathcal{D}$ around 0. Then there exists an open disk $\mathcal{D}_1$ around 0, $\mathcal{D}_1\subset\mathcal{D}$ and a real analytic family 
\[\{\mathsf{f}_u : \mathsf{U}_0\Gamma \rightarrow \mathbb{R}\}_{u\in \mathcal{D}_1}\] 
of positive H\"older continuous functions such that the function $\mathsf{f}_u$ is Li\u vsic cohomologous to the function $\mathfrak{f}_{\rho_u}$ for all $u\in\mathcal{D}_1$.
\end{proposition}
\begin{proof}
We start by constructing the following line bundle:
\begin{equation}
\mathcal{B} \defeq \left\lbrace \left( (X,v),P_{X,v,v^+}, P_{X,v,v^-}\right) \mid (X,v) \in \mathsf{U}\mathbb{A} \right\rbrace
\end{equation}
over $\mathsf{G}/\mathsf{L}$. Now using proposition \ref{marano} and theorem \ref{lanal} we get that there exists a sub-disk $\mathcal{D}_0\subset\mathcal{D}$, containing 0, and $\mu$-H\"older transversely real analytic maps,
\begin{align}
(\xi^+,\xi^-) : \mathcal{D}_0\times\partial_{\infty}\Gamma^{(2)} \rightarrow \mathsf{G}/\mathsf{L}.
\end{align}
Let us consider the the projection map,
\begin{align}
\pi: \mathcal{D}_0\times\widetilde{\mathsf{U}_0\Gamma} &\rightarrow \mathcal{D}_0\times\partial_{\infty}\Gamma^{(2)}\\
\notag \left(u,(x,y,t)\right) &\mapsto \left(u,(x,y)\right)
\end{align}
and note that the map $(\xi^+,\xi^-)\circ\pi$ is $\mu$-H\"older transversely real analytic. We take the pullback of this map to define a $\mu$-H\"older transversely real analytic bundle $\tilde{\mathcal{B}} \defeq \left((\xi^+,\xi^-)\circ\pi\right)^*\mathcal{B}$ over $\mathcal{D}_0\times\widetilde{\mathsf{U}_0\Gamma}$. The free group $\Gamma$ acts on this bundle as follows:
\begin{align*}
&\gamma.\left(u,(x,y,t),\left((X,v),\xi^+_u(x,y,t), \xi^-_u(x,y,t)\right)\right)\\ 
&\defeq \left(u,\gamma.(x,y,t),\left(\left(\rho_u(\gamma)X,\mathtt{L}_{\rho_u(\gamma)}v\right),\xi^+_u(\gamma(x,y,t)), \xi^-_u(\gamma(x,y,t))\right) \right)
\end{align*}
We observe that the action of $\Gamma$ gives rise to a quotient bundle $\Gamma\backslash\tilde{\mathcal{B}}$ over $\mathcal{D}_0\times\mathsf{U}_0\Gamma$. Let $\sigma$ be a $\mu$-H\"older transversely real analytic section of this bundle and let $\tilde{\sigma}$ be its lift onto $\mathcal{D}_0\times\widetilde{\mathsf{U}_0\Gamma}$. Let $\{\tilde{\psi}_t\}_{t\in\mathbb{R}}$ be the flow on $\mathcal{D}_0\times\widetilde{\mathsf{U}_0\Gamma}$ such that $\tilde{\psi}_t\left(u,(x,y,t_0)\right)\defeq\left(u,(x,y,t+t_0)\right)$. Also let $\pi_1,\pi_2$ denote the map which sends $\left( (X,v),P_{X,v,v^+}, P_{X,v,v^-}\right)$ to $X$ and $v$ respectively. We observe that for all real number $t$
\begin{align}
\pi_1\tilde{\psi}_t^*\tilde{\sigma}(u,(x,y,t_0))= &\pi_1\tilde{\sigma}(u,(x,y,t_0))\\
\notag &+k_t(u,(x,y,t_0))\pi_2\tilde{\sigma}(u,(x,y,t_0))
\end{align}
where $k_t : \mathcal{D}_0\times\widetilde{\mathsf{U}_0\Gamma} \rightarrow \mathbb{R}$ is a $\mu$-H\"older transversely real analytic function and for all real number $t$
\begin{equation}\label{pi2}
\pi_2\tilde{\psi}_t^*\tilde{\sigma}(u,(x,y,t_0))=\pi_2\tilde{\sigma}(u,(x,y,t_0)).
\end{equation}
Let $t_\gamma$ be the period of the geodesic $\{(\gamma^-,\gamma^+,t)\mid t\in\mathbb{R}\}$ fixed by $\gamma$ in $\Gamma$. We further notice that
\begin{align*}
\mathtt{L}_{\rho_u}(\gamma)\pi_2\tilde{\sigma}(u,(\gamma^-,\gamma^+,t_0))&=\pi_2\tilde{\sigma}(u,\gamma(\gamma^-,\gamma^+,t_0))\\
&=\pi_2\tilde{\sigma}(u,(\gamma^-,\gamma^+,t_0+t_\gamma))\\
&=\pi_2\tilde{\sigma}(u,(\gamma^-,\gamma^+,t_0)).
\end{align*}
We also recall that $\pi_2\tilde{\sigma}(0,(\gamma^-,\gamma^+,t_0))=\nu_{\rho_0}\left(\gamma^-,\gamma^+\right)$. Therefore we deduce that
\begin{align}\label{limnu}
\pi_2\tilde{\sigma}(u,(\gamma^-,\gamma^+,t_0))=\nu_{\rho_u}\left(\gamma^-,\gamma^+\right).
\end{align}
Furthermore, for all real number $t_0$ and $t$ we have,
\begin{align*}
k_{t+t_\gamma}&(u,(\gamma^-,\gamma^+,t_0))\pi_2\tilde{\sigma}(u,(\gamma^-,\gamma^+,t_0))\\
= &(k_t(u,(x,y,t_0))+\alpha_{\rho_u}(\gamma))\pi_2\tilde{\sigma}(u,(\gamma^-,\gamma^+,t_0)).
\end{align*}
Therefore we get that for all real number $t_0$
\begin{equation}\label{period}
k_{t+t_\gamma}(u,(\gamma^-,\gamma^+,t_0)) = k_t(u,(\gamma^-,\gamma^+,t_0)) + \alpha_{\rho_u}(\gamma).
\end{equation}
We also note that for all real number $t$ and $t^\prime$ we have
\begin{align*}
k_{t+t^\prime}&(u,(x,y,t_0))\pi_2\tilde{\sigma}(u,(x,y,t_0))\\
=\ &k_t(u,(x,y,t_0+t^\prime))\pi_2\tilde{\sigma}(u,(x,y,t_0+t^\prime))\\
&+k_{t^\prime}(u,(x,y,t_0))\pi_2\tilde{\sigma}(u,(x,y,t_0)).
\end{align*}
And using equation \ref{pi2} we get that
\begin{align}\label{cocycle}
k_{t+t^\prime}&(u,(x,y,t_0))=k_t(u,(x,y,t_0+t^\prime))+k_{t^\prime}(u,(x,y,t_0)).
\end{align}
Now we fix some real number $r>0$ and define
\begin{align*}
\mathfrak{K}_t(u,(x,y,t_0)) \defeq \log\left(\frac{\int_t^{r+t}\exp(k_s(u,(x,y,t_0)))ds}{\int_0^{r}\exp(k_s(u,(x,y,t_0)))ds}\right).
\end{align*}
Using equation \ref{period} we get that
\begin{align}
\mathfrak{K}_{t+t_\gamma}(u,(\gamma^-,\gamma^+,t_0))= \mathfrak{K}_t(u,(\gamma^-,\gamma^+,t_0)) + \alpha_{\rho_u}(\gamma).
\end{align}
Moreover, using equation \ref{cocycle} we get that
\begin{align}\label{cocycle1}
\mathfrak{K}_{t+t^\prime}&(u,(x,y,t_0))=\mathfrak{K}_t(u,(x,y,t_0+t^\prime))+\mathfrak{K}_{t^\prime}(u,(x,y,t_0)).
\end{align}
Finally we define
\begin{align}
{f}_u(x,y,t_0) \defeq \left.\frac{\partial}{\partial t}\right|_{t=0}\mathfrak{K}_t(u,(x,y,t_0)).
\end{align}
We notice that
\begin{align*}
&\left.\frac{\partial}{\partial t}\right|_{t=0}\mathfrak{K}_t(u,(x,y,t_0))=\left.\frac{\partial}{\partial t}\right|_{t=0}\log\left(\frac{\int_t^{r+t}\exp(k_s(u,(x,y,t_0)))ds}{\int_0^{r}\exp(k_s(u,(x,y,t_0)))ds}\right)\\
&=\left.\frac{\partial}{\partial t}\right|_{t=0}\log\left(\int_t^{r+t}\exp(k_s(u,(x,y,t_0)))ds\right)\\
&=\frac{\left.\frac{\partial}{\partial t}\right|_{t=0}\int_0^t\left(\exp(k_{s+r}(u,(x,y,t_0)))- \exp(k_s(u,(x,y,t_0)))\right)ds}{\int_0^r\exp(k_s(u,(x,y,t_0)))ds}\\
&=\frac{\exp(k_r(u,(x,y,t_0)))- \exp(k_0(u,(x,y,t_0)))}{\int_0^r\exp(k_s(u,(x,y,t_0)))ds}.
\end{align*}
Therefore ${f}_u(x,y,t_0)$ is also $\mu$-H\"older transeversely real analytic. Moreover, using equation \ref{cocycle1} one gets
\begin{align*}
\left.\frac{\partial}{\partial t}\right|_{t=0}\mathfrak{K}_t(u,(x,y,t_0))=\left.\frac{\partial}{\partial t}\right|_{t=t_0}\mathfrak{K}_t(u,(x,y,0)).
\end{align*}
Hence we have
\begin{align}
\int_0^{t_\gamma}{f}_u(\gamma^-,\gamma^+,s)ds &= \int_0^{t_\gamma}\left.\frac{\partial}{\partial t}\right|_{t=s}\mathfrak{K}_t(u,(\gamma^-,\gamma^+,0))ds\\
\notag &=\mathfrak{K}_{t_\gamma}(u,(\gamma^-,\gamma^+,0))-\mathfrak{K}_0(u,(\gamma^-,\gamma^+,0))\\
\notag &=\alpha_{\rho_u}(\gamma).
\end{align}
Therefore if $u\in\mathcal{D}_0$ then
\[\int_\gamma{f}_u=\int_\gamma\mathfrak{f}_{\rho_u}\]
for all $\gamma\in\Gamma$. Now using theorem 3.3 of \cite{pressure metric} (originally proved by Li\u vsic in \cite{Liv}) we deduce that ${f}_u$ is Li\u vsic cohomologous to the positive H\"older function $\mathfrak{f}_{\rho_u}$ for all $u\in\mathcal{D}_0$. Therefore for any flow invariant measure $\mathfrak{m}$ on $\mathsf{U}_0\Gamma$ we have
\[\int{f}_u\mathsf{d}\mathfrak{m}=\int\mathfrak{f}_{\rho_u}\mathsf{d}\mathfrak{m}>0.\]
Now using lemma A.1 and lemma A.2 of \cite{geodesic} and transverse analyticity of ${f}_u$ we derive that there exists a neighborhood $\mathcal{D}_1\subset\mathcal{D}_0$ and there exists a real number $T>0$ such that for all $u\in\mathcal{D}_1$
\[{f}^T_u(x,y,t_0)\defeq\frac{1}{T}\int_{0}^{T}{f}_u(x,y,t_0+s)ds>0.\]
Now we finish our proof by considering the collection 
\[\{\mathsf{f}_u\defeq{f}^T_u\mid u\in\mathcal{D}_1\}\]
and noticing that it satisfies all the required properties.
\end{proof}

\subsection{Deformation of the cross ratio}\label{cross}

In this section we obtain a formula for the variation of the cross ratio which is similar in taste to the theorem \ref{gm}. We start by stating an alternative version of the proposition 10.4 from \cite{pressure metric}.
\begin{proposition}{[Bridgeman, Canary, Labourie, Sambarino]}\label{cratio}
Let $\varrho$ be an element of $\mathsf{Hom}_{\hbox{\tiny $\mathrm{S}$}}(\Gamma, \mathsf{SO}^0(2,1))$. Then
\[\lim_{n\to\infty}\left(\ell_\varrho(\gamma^n\eta^n) - \ell_\varrho(\gamma^n)- \ell_\varrho(\eta^n)\right) = \log\mathsf{b}_\varrho(\eta^-,\gamma^-,\gamma^+,\eta^+)\]
where $\ell_\rho(\gamma)$ is the length of the closed geodesic corresponding to $\varrho(\gamma)$.
\end{proposition}

\begin{lemma}\label{x0}
Let $\{\rho_t\}$ be a smooth path in $\mathsf{Hom}_{\hbox{\tiny $\mathrm{M}$}}(\Gamma, \mathsf{G})$. Then the following holds
\[\lim_{n \to \infty} \left.\frac{d}{dt}\right|_{t = 0} \nu_{\rho_t}\left((\gamma^n\eta^n)^-,(\gamma^n\eta^n)^+\right)= \left.\frac{d}{dt}\right|_{t = 0} \nu_{\rho_t}\left(\eta^-,\gamma^+\right).\]
Moreover, the rate of convergence is exponential.
\end{lemma}
\begin{proof}
As $\{\rho_t\}$ is a path in $\mathsf{Hom}_{\hbox{\tiny $\mathrm{M}$}}(\Gamma, \mathsf{G})$ we can consider it as a path in $\{\rho_u\}_{u\in\mathcal{D}}$, a complex analytic family in $\mathsf{Hom}(\Gamma, \mathsf{G}_\mathbb{C})$ parametrized by a complex disk $\mathcal{D}$ around $0$. Using theorem \ref{canal} we get that the limit maps $\xi^+$ and $\xi^-$ are $\mu$-H\"older transversely complex analytic. Hence 
\begin{align*}
\left\{\xi^+\left((\gamma^n\eta^n)^-\right)\right\}_{n\in\mathbb{N}}
\end{align*}
is a sequence of complex analytic maps converging to $\xi^+\left(\eta^-\right)$ on $\mathcal{D}$. Moreover, as $(\gamma^n\eta^n)^-$ converges to $\eta^-$ at an exponential rate and the limit map $\xi^+$ is $\mu$-H\"older we get that the rate of convergence is exponential. Now as
\[\left\{\xi^+\left((\gamma^n\eta^n)^-\right)\right\}_{n\in\mathbb{N}}\]
is a sequence of complex analytic functions on $\mathcal{D}$ converging exponentially to $\xi^+\left(\eta^-\right)$, using Cauchy's integral formula we get that the derivative of the sequence is also converging exponentially to the derivative of $\xi^+\left(\eta^-\right)$. Now restricting the limit maps on the real part we get that
\begin{align*}
\lim_{n \to \infty} \left.\frac{d}{dt}\right|_{t = 0} \xi^+_{\rho_t}\left((\gamma^n\eta^n)^-\right) = \left.\frac{d}{dt}\right|_{t = 0} \xi^+_{\rho_t}\left(\eta^-\right)
\end{align*}
with the convergence rate being exponential. Similarly we get that
\begin{align*}
\lim_{n \to \infty} \left.\frac{d}{dt}\right|_{t = 0} \xi^-_{\rho_t}\left((\gamma^n\eta^n)^+\right) = \left.\frac{d}{dt}\right|_{t = 0} \xi^-_{\rho_t}\left(\gamma^+\right)
\end{align*}
where the convergence rate is exponential.

Let $\tilde{\pi}_2$ be the projection from $\mathsf{U}\mathbb{A}$ onto $\mathsf{S}^{1}$. We note that $\tilde{\pi}_2$ gives rise to a projection map
\[\pi_2: \mathsf{G}/\mathsf{L}\longrightarrow\mathsf{S}^{1}.\]   
We conclude our proof by recalling from equation \ref{limnu} that
\[\nu_\rho\left(\eta^-,\gamma^+\right) = \pi_2\circ(\xi^+_\rho,\xi^-_\rho)\left(\eta^-,\gamma^+\right).\]
\end{proof}
\begin{proposition}\label{main}
Let $\{\rho_t\}$ be a smooth path in $\mathsf{Hom}_{\hbox{\tiny $\mathrm{M}$}}(\Gamma, \mathsf{G})$. Also let $X_{\rho_t(\gamma)}$ be any point on the unique affine line fixed by $\rho_t(\gamma)$ where $\gamma$ is in $\Gamma$. Then for all coprime $\gamma, \eta$ in $\Gamma$ we have 
\begin{align*}
&\lim_{n \to \infty} \left(\alpha_{\rho_t}(\gamma^n\eta^n) - \alpha_{\rho_t}(\gamma^n) - \alpha_{\rho_t}(\eta^n) \right)\\ 
\notag&= \left\langle X_{\rho_t(\gamma)} - X_{\rho_t(\eta)}\mid\nu_{\rho_t}\left(\eta^-, \gamma^+\right)+\nu_{\rho_t}\left(\eta^+, \gamma^-\right)\right\rangle,\\
&\left.\frac{d}{dt}\right|_{t = 0} \left\langle X_{\rho_t(\gamma)} - X_{\rho_t(\eta)}\mid\nu_{\rho_t}\left(\eta^-, \gamma^+\right)+\nu_{\rho_t}\left(\eta^+, \gamma^-\right)\right\rangle\\
\notag&=\lim_{n \to \infty} \left.\frac{d}{dt}\right|_{t = 0} \left(\alpha_{\rho_t}(\gamma^n\eta^n) - \alpha_{\rho_t}(\gamma^n) - \alpha_{\rho_t}(\eta^n) \right).
\end{align*}
\end{proposition}
\begin{proof}
We begin the proof by mentioning that the first identity is a variation of an identity worked out by Charette--Drumm in page 12 of \cite{cd}. In fact I use the same method used by them to compute both the identities.

Let ${l}_{\rho(\eta)}$ be the unique affine line fixed by $\rho(\eta)$ and let $\nu_{\rho}\left(\gamma^-, \gamma^+\right)^\perp$ be the plane which is perpendicular to the unit vector $\nu_{\rho}\left(\gamma^-, \gamma^+\right)$ in the Lorentzian metric. We note that $v^\perp\cap\mathcal{C}$ is a disjoint union of two half lines where $\mathcal{C}$ is the upper half of $\mathsf{S}^0\backslash\{0\}$. We choose $v^\pm\in v^\perp\cap\mathcal{C}$ such that $(v^+,v,v^-)$ gives the same orientation as $(v_0^+,v_0,v_0^-)$. We denote the affine plane which is parallel to the plane generated by $v^-$ and $\nu_{\rho}\left(\gamma^-, \gamma^+\right)$ and which contains ${l}_{\rho(\gamma)}$ by ${l}^-_{\rho(\gamma)}$. As $\gamma$ and $\eta$ are coprime we have that ${l}_{\rho(\eta)}$ intersects ${l}^-_{\rho(\gamma)}$ in a unique point $Q_\rho$. Also let $R$ be the point on ${l}_{\rho(\gamma)}$ such that 
\[\langle R-Q_\rho,\nu_\rho(\gamma)\rangle = 0\]
where $\nu_\rho(\gamma)\defeq\nu_\rho\left(\gamma^-, \gamma^+\right)$. We note that as $Q_\rho\in{l}_{\rho(\eta)}$ we have
\[Q_\rho - \rho(\eta)^{-n}Q_\rho=\alpha_\rho(\eta^n)\nu_\rho(\eta)\]
and as $R\in{l}_{\rho(\gamma)}$ we have
\[\rho(\gamma)^nR - R = \alpha_{\rho}(\gamma^n)\nu_\rho(\gamma).\]
\begin{center}
\text{}\\
\includegraphics[height=6cm]{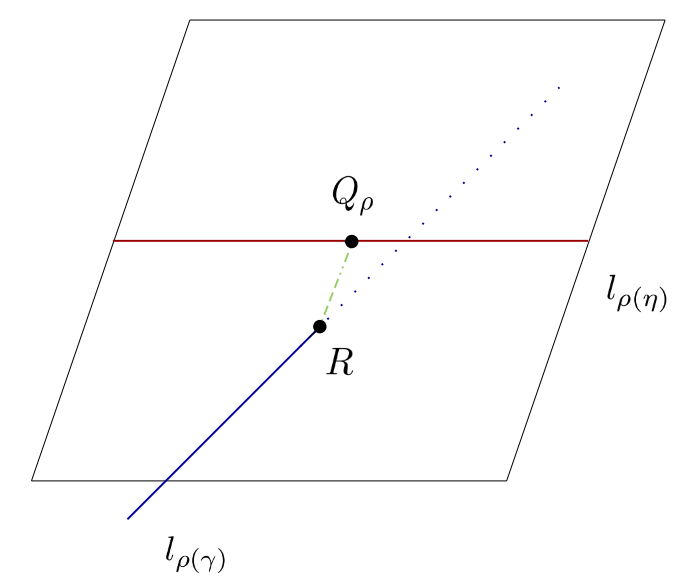}
\end{center}
Now we observe that
\begin{align*}
\alpha_\rho(\gamma^n\eta^n) =&\ \langle\rho(\gamma)^nQ_\rho-\rho(\eta)^{-n}Q_\rho\mid \nu_\rho(\gamma^n\eta^n)\rangle\\
=&\ \langle\rho(\gamma)^nQ_\rho-\rho(\gamma)^nR-(Q_\rho-R)\mid \nu_\rho(\gamma^n\eta^n)\rangle \\
&+\ \langle (Q_\rho - \rho(\eta)^{-n}Q_\rho) + (\rho(\gamma)^nR - R) \mid \nu_\rho(\gamma^n\eta^n)\rangle\\
 =&\ \langle\left(\mathtt{L}_\rho(\gamma)^n-\mathbb{I}\right)(Q_\rho-R)\mid \nu_\rho(\gamma^n\eta^n)\rangle\\
 &+ \langle \alpha_{\rho}(\gamma^n)\nu_\rho(\gamma) + \alpha_{\rho}(\eta^n)\nu_\rho(\eta)\mid \nu_\rho(\gamma^n\eta^n)\rangle .
\end{align*}
We observe that the vector $(Q_\rho-R)$ is an eigenvector of $\mathtt{L}_\rho(\gamma)$ with eigenvalue $\lambda_\rho(\gamma)$ such that $|\lambda_\rho(\gamma)|<1$. Therefore we get that
\begin{align*}
\alpha_{\rho}(\gamma^n\eta^n) =&\ \left(\lambda_\rho(\gamma)^n-1\right)\langle Q_\rho-R\mid \nu_\rho(\gamma^n\eta^n)\rangle\\
 &+ \langle \alpha_\rho(\gamma^n)\nu_\rho(\gamma) + \alpha_\rho(\eta^n)\nu_\rho(\eta)\mid \nu_\rho(\gamma^n\eta^n)\rangle.
\end{align*}
We recall that
\[\left\langle \nu_\rho(\gamma)\mid \nu_\rho(\eta^-, \gamma^+)\right\rangle = 1 = \left\langle \nu_\rho(\eta)\mid \nu_\rho(\eta^-, \gamma^+)\right\rangle.\]
Hence we get
\begin{align*}
\alpha&_\rho(\gamma^n\eta^n)-\alpha_\rho(\gamma^n)-\alpha_\rho(\eta^n)\\
=&\ \left(\lambda_\rho(\gamma)^n-1\right)\langle Q_\rho-R\mid \nu_\rho(\gamma^n\eta^n)\rangle\\
 &+ \alpha_\rho(\gamma^n)\langle \nu_\rho(\gamma)\mid \nu_\rho(\gamma^n\eta^n)-\nu_\rho(\eta^-, \gamma^+)\rangle\\
 &+ \alpha_\rho(\eta^n)\langle \nu_\rho(\eta)\mid \nu_\rho(\gamma^n\eta^n)-\nu_\rho(\eta^-, \gamma^+)\rangle.
\end{align*}
Now using the fact that $\nu_\rho(\gamma^n\eta^n)$ converges exponentially to $\nu_\rho(\eta^-, \gamma^+)$, while $\alpha_\rho(\gamma^n)$ has polynomial growth and the fact that $|\lambda_\rho(\gamma)|<1$ we obtain
\begin{align*}
\lim_{n\to\infty}(\alpha_\rho(\gamma^n\eta^n)-\alpha_\rho(\gamma^n)-\alpha_\rho(\eta^n)) = -\langle Q_\rho-R \mid \nu_\rho(\eta^-, \gamma^+)\rangle.
\end{align*}
Moreover, using lemma \ref{x0} and the fact that $|\lambda_\rho(\gamma)|<1$ we deduce that 
\begin{align*}
&\lim_{n\to\infty}\left.\frac{d}{dt}\right|_{t = 0}(\alpha_{\rho_t}(\gamma^n\eta^n)-\alpha_{\rho_t}(\gamma^n)-\alpha_{\rho_t}(\eta^n))\\
 &= -\left.\frac{d}{dt}\right|_{t = 0}\langle Q_{\rho_t}-R\mid \nu_{\rho_t}(\eta^-, \gamma^+)\rangle.
\end{align*}
Finally, we conclude by observing that
\begin{align*}
&\langle R- Q_\rho \mid \nu_\rho(\eta^-, \gamma^+)\rangle = \langle X_{\rho(\gamma)}-X_{\rho(\eta)}\mid\nu_\rho(\eta^-, \gamma^+)+\nu_\rho(\eta^+, \gamma^-)\rangle
\end{align*}
where $X_{\rho(\gamma)}\in{l}_{\rho(\gamma)}$ and $X_{\rho(\eta)}\in{l}_{\rho(\eta)}$ are any two points for $\gamma,\eta\in\Gamma$.
\end{proof}
\begin{theorem}\label{dcr}
Let $\{\varrho_t\}$ be a smooth path in $\mathsf{Hom}_{\hbox{\tiny $\mathrm{S}$}}(\Gamma, \mathsf{SO}^0(2,1))$ such that $\rho\defeq(\varrho_0,\dot{\varrho}_0)\in\mathsf{Hom}_{\hbox{\tiny $\mathrm{M}$}}(\Gamma, \mathsf{G})$ where $\dot{\varrho}_0\defeq\left.\frac{d}{dt}\right|_{t=0}\varrho_t$. Then we have
\begin{align*}
\langle &X_{\rho(\gamma)}-X_{\rho(\eta)}\mid\nu_\rho(\eta^-, \gamma^+)+\nu_\rho(\eta^+, \gamma^-)\rangle\\
&=\left.\frac{d}{dt}\right|_{t=0}\log\mathsf{b}_{\varrho_t}(\eta^-,\gamma^-,\gamma^+,\eta^+)
\end{align*}
where $X_{\rho(\gamma)}$ is any point on the unique affine line fixed by $\rho(\gamma)$ and $X_{\rho(\eta)}$ is any point on the unique affine line fixed by $\rho(\eta)$.
\end{theorem}
\begin{proof}
The result follows from using theorem \ref{gm}, proposition \ref{cratio}, lemma \ref{x0} and proposition \ref{main}.
\end{proof}

\section{Properties of the pressure metric}
In this section we will introduce the pressure metric and prove its properties in the context of Margulis Space Times. We note that subsection 5.1 and 5.2 closely follows \cite{pressure metric} and have been included here to make the exposition complete. Moreover, we also mention that the study of pressure metric and thermodynamical formalism was originally developed by Bowen, Bowen--Ruelle, Parry--Pollicott, Pollicott and Ruelle in \cite{Bow1}, \cite{Bow}, \cite{BR}, \cite{PP}, \cite{P}, \cite{R}.

\subsection{The thermodynamic mapping}

Let $\rho\in\mathsf{Hom}_{\hbox{\tiny $\mathrm{M}$}}(\Gamma, \mathsf{G})$ and let $h_{\mathfrak{f}_{\rho}}$ be the topological entropy of the reparametrized flow on $\mathsf{U}_0\Gamma$ corresponding to the reparametrization $\mathfrak{f}_{\rho}$. By theorem 1.0.1 of \cite{me} we know that the geodesic flow on $\mathsf{U}_{\hbox{\tiny $\mathrm{rec}$}}\mathsf{M}_{\rho}$ is metric Anosov. Hence by using proposition 3.5 of \cite{pressure metric} (originally proved by Bowen\cite{Bow1} and Pollicott\cite{P}) we deduce that $h_{\mathfrak{f}_{\rho}}$ is finite and positive and moreover,
\begin{align*}
h_{\mathfrak{f}_{\rho}} = \lim_{T \to \infty}\frac{1}{T} \log \left(\# \left\{[\gamma] \in \mathsf{O}(\Gamma) \mid \int_{\gamma}\mathfrak{f}_{\rho} \leqslant T \right\}\right)
\end{align*}
where $\mathsf{O}(\Gamma)$ is the set of closed orbits of $\mathsf{U}_0\Gamma$.
We also recall that for all $\gamma\in\Gamma$
\[\int_{\gamma}\mathfrak{f}_{\rho} = \alpha_{\rho}(\gamma).\]
Therefore we see that $h_{\mathfrak{f}_{\rho}}$ only depends on the Li\u vsic cohomology class of $\mathfrak{f}_{\rho}$. Hence we denote $h_{\mathfrak{f}_{\rho}}$ by $h_{\rho}$ and we get that
\begin{align}\label{hh}
h_{\rho} = \lim_{T \to \infty}\frac{1}{T} \log \left(\# \left\{[\gamma] \in\mathsf{O}(\Gamma) \mid \alpha_{\rho}(\gamma) \leqslant T \right\}\right).
\end{align}
\begin{lemma}\label{ent}
The following map is analytic:
\begin{align*}
h : \mathsf{Hom}_{\hbox{\tiny $\mathrm{M}$}}(\Gamma, \mathsf{G}) &\longrightarrow \mathbb{R} \\
\notag \rho &\longmapsto h_{\rho}
\end{align*}
\end{lemma}
\begin{proof}
The result follows from proposition 3.12 of \cite{pressure metric}, proposition \ref{ranal} and theorem 1.0.1 of \cite{me}.
\end{proof}
We recall that the Gromov flow $\psi$ on the compact metric space $\mathsf{U}_0\Gamma$ is H\"older. Now using lemma 3.1 of \cite{pressure metric} and proposition \ref{ranal} we deduce that the pressure of the map $-h_\rho\mathsf{f}_\rho$ is zero with respect to the Gromov flow $\psi$. Let us denote the space of all Li\u vsic cohomology classes of pressure zero functions on $\mathsf{U}_0\Gamma$ by $\mathcal{H}(\mathsf{U}_0\Gamma)$.
\begin{definition}
We define the $\textit{Thermodynamic mapping}$ as follows,
\begin{align*}
\mathfrak{T} : \mathsf{Hom}(\Gamma, \mathsf{G}) &\longrightarrow \mathcal{H}(\mathsf{U}_0\Gamma)\\
\notag \rho &\longmapsto [-h_{\rho}\mathsf{f}_{\rho}].
\end{align*}
\end{definition}
\begin{lemma}
The map $\mathfrak{T}$ is analytic.
\end{lemma}
\begin{proof}
The result follows from proposition \ref{ranal} and the fact that the entropy funtion is also analytic.
\end{proof}

\subsection{The pressure metric}

Let $I(f,g)$ be the $\textit{intersection number}$ of the two reparametrizations $f$ and $g$. As our flow is metric Anosov, using theorem 3.7 of \cite{pressure metric} (originally proved by Bowen\cite{Bow1} and Pollicott\cite{P}) and equation (7) of \cite{pressure metric} we get that
\begin{align*}
I(\mathsf{f}_{\rho_1},\mathsf{f}_{\rho_2})=\lim_{T\to\infty}\frac{1}{\#R_T(\rho_1)}\sum_{[\gamma]\in R_T(\rho_1)}\frac{\alpha_{\rho_2}(\gamma)}{\alpha_{\rho_1}(\gamma)}
\end{align*}
where $R_T(\rho_1)\defeq \left\{[\gamma] \in\mathsf{O}(\Gamma) \mid \alpha_{\rho_1}(\gamma) \leqslant T \right\}$. And using proposition 3.12 of \cite{pressure metric} and proposition \ref{ranal} we notice that the map $I$ is analytic. Let us define 
\[J_{\rho_1}(\rho_2)\defeq I(\rho_1,\rho_2)\frac{h_{\rho_2}}{h_{\rho_1}}.\]

\begin{proposition}\label{p}
The following statements are true:
\begin{enumerate}
\item for all $\rho_1,\rho_2\in\mathsf{Hom}_{\hbox{\tiny $\mathrm{M}$}}(\Gamma, \mathsf{G})$ we have $J_{\rho_1}(\rho_2)\geqslant1$,
\item if $\rho_1,\rho_2\in\mathsf{Hom}_{\hbox{\tiny $\mathrm{M}$}}(\Gamma, \mathsf{G})$ and $J_{\rho_1}(\rho_2)=1$ then there exists a positive real number $c$ such that for all $\gamma\in\Gamma$
\[c\alpha_{\rho_1}(\gamma)=\alpha_{\rho_2}(\gamma),\]
\item if $\{\rho_t\}_{t\in \mathtt{I}}\subset\mathsf{Hom}_{\hbox{\tiny $\mathrm{M}$}}(\Gamma, \mathsf{G})$ is a smooth path parametrized by an open interval $\mathtt{I}\subset\mathbb{R}$ containing 0 then 
\[\left.\frac{\partial^2}{\partial t^2}\right|_{t=0}J_{\rho_0}(\rho_t)=0\]
if and only if 
\[\left.\frac{d}{dt}\right|_{t=0} h_{\rho_t}\mathsf{f}_{\rho_t}\simeq0\] 
in Li\u vsic cohomology.
\end{enumerate}
\end{proposition}
\begin{proof}
The result follows from propositions 3.8, 3.9 and 3.11 of \cite{pressure metric}. We note that proposition 3.9 of \cite{pressure metric} was originally proved by Parry--Pollicott\cite{PP} and Ruelle\cite{R} and proposition 3.11 of \cite{pressure metric} is a generalization of an earlier result due to Bonahon\cite{bona}.
\end{proof}
\begin{definition}
Let $\rho\in\mathsf{Hom}_{\hbox{\tiny $\mathrm{M}$}}(\Gamma, \mathsf{G})$ and let $v,w\in\mathsf{T}_\rho\mathsf{Hom}_{\hbox{\tiny $\mathrm{M}$}}(\Gamma, \mathsf{G})$. We define
\[\mathtt{P}_\rho(v,w)\defeq\mathsf{D}^2_\rho J_\rho(v,w).\]
The map $\mathtt{P}$ is called the \textit{pressure form} on $\mathsf{Hom}_{\hbox{\tiny $\mathrm{M}$}}(\Gamma, \mathsf{G})$.
\end{definition}

\begin{remark}\label{nonneg}
We notice that by proposition \ref{p} the pressure form $\mathtt{P}$ on $\mathsf{Hom}_{\hbox{\tiny $\mathrm{M}$}}(\Gamma, \mathsf{G})$ is non-negative definite.
\end{remark}

\subsection{Vectors with pressure norm zero}

In this subsection we will describe the zero vectors of the pressure norm.
\begin{proposition}\label{z1}
Let $\{\rho_t\}$ be a smooth path in $\mathsf{Hom}_{\hbox{\tiny $\mathrm{M}$}}(\Gamma, \mathsf{G})$ with $\left.\frac{d}{dt}\right|_{t=0} \rho_t = v$. If $\mathtt{P}_\rho(v,v) = 0$ and $\left.\frac{d}{dt}\right|_{t=0} h_{\rho_t} = 0$ then for all $\gamma$ in $\Gamma$
\begin{align*}
\left.\frac{d}{dt}\right|_{t=0} \alpha_{\rho_t}(\gamma) = 0.
\end{align*}
\end{proposition}
\begin{proof}
We start by using proposition \ref{p} and notice that $\left.\frac{d}{dt}\right|_{t=0} h_{\rho_t}\mathsf{f}_{\rho_t}$ is Li\u vsic cohomologous to zero. Hence for all closed orbits $[\gamma]\in\mathsf{O}(\Gamma)$ we have that
\[\int_\gamma\left.\frac{d}{dt}\right|_{t=0} h_{\rho_t}\mathsf{f}_{\rho_t} =0.\]
Now we observe that
\begin{align*}
0&=\int_\gamma\left.\frac{d}{dt}\right|_{t=0} h_{\rho_t}\mathsf{f}_{\rho_t}= \int_\gamma\left(\left.\frac{d}{dt}\right|_{t=0} h_{\rho_t}\right)\mathsf{f}_{\rho_0} + \int_\gamma h_{\rho_0}\left(\left.\frac{d}{dt}\right|_{t=0} \mathsf{f}_{\rho_t}\right)\\
&= h_{\rho_0}\int_\gamma \left.\frac{d}{dt}\right|_{t=0} \mathsf{f}_{\rho_t} = h_{\rho_0}\left.\frac{d}{dt}\right|_{t=0}\int_\gamma\mathsf{f}_{\rho_t}=h_{\rho_0}\left.\frac{d}{dt}\right|_{t=0} \alpha_{\rho_t}(\gamma).
\end{align*}
We conclude by recalling that the entropy $h_{\rho_0}$ is positive and hence our result follows.
\end{proof}

\begin{lemma}\label{z2}
If for all $\gamma\in\Gamma$ we have $\left.\frac{d}{dt}\right|_{t=0} \alpha_{\rho_t}(\gamma) = 0$ then for all $\gamma, \eta\in\Gamma$ we have
\[\left.\frac{d}{dt}\right|_{t = 0}\mathsf{b}_{\rho_t}\left(\eta^+,\gamma^-, \gamma^+, \eta^-\right)=0.\]
\end{lemma}
\begin{proof}
Using proposition \ref{main} we get that 
\[\left.\frac{d}{dt}\right|_{t = 0} \left\langle X_{\rho_t(\gamma)} - X_{\rho_t(\eta)}\mid\nu_{\rho_t}\left(\eta^-, \gamma^+\right)+\nu_{\rho_t}\left(\eta^+, \gamma^-\right)\right\rangle=0\]
and also
\[\left.\frac{d}{dt}\right|_{t = 0} \left\langle X_{\rho_t(\gamma)} - X_{\rho_t(\eta)}\mid\nu_{\rho_t}\left(\eta^+, \gamma^+\right)+\nu_{\rho_t}\left(\eta^-, \gamma^-\right)\right\rangle=0.\]
Now using identities \ref{b1}, \ref{b2} and \ref{b3} we get that
\begin{align*}
&\mathsf{b}_{\rho_t}\left(\eta^+,\gamma^-,\gamma^+,\eta^-\right)\left(\nu_{\rho_t}\left(\eta^+, \gamma^+\right)+\nu_{\rho_t}\left(\eta^-, \gamma^-\right)\right)\\
&= \mathsf{b}_{\rho_t}\left(\eta^-,\gamma^-,\gamma^+,\eta^+\right)\left(\nu_{\rho_t}\left(\eta^-, \gamma^+\right)+\nu_{\rho_t}\left(\eta^+, \gamma^-\right)\right)\\
&= \left(1-\mathsf{b}_{\rho_t}\left(\eta^+,\gamma^-,\gamma^+,\eta^-\right)\right)\left(\nu_{\rho_t}\left(\eta^-, \gamma^+\right)+\nu_{\rho_t}\left(\eta^+, \gamma^-\right)\right).
\end{align*}
Therefore we deduce that 
\[\left.\frac{d}{dt}\right|_{t = 0}\mathsf{b}_{\rho_t}\left(\eta^+,\gamma^-,\gamma^+,\eta^-\right)=0\]
for all $\gamma,\eta\in\Gamma$.
\end{proof}

\begin{proposition}\label{final}
Let $\{\rho_t\}$ be a smooth path in $\mathsf{Hom}_{\hbox{\tiny $\mathrm{M}$}}(\Gamma, \mathsf{G})$ with $\left.\frac{d}{dt}\right|_{t=0} \rho_t = \dot{\rho}_0.$ If $\mathtt{P}_{\rho_0}(\dot{\rho}_0,\dot{\rho}_0) = 0$ and $\left.\frac{d}{dt}\right|_{t=0} h_{\rho_t} = 0$ then 
\[[\dot{\rho}_0] = 0\] 
in $\mathsf{H}^1_{\rho_0}\left(\Gamma, \mathfrak{g}\right)$ where $\mathfrak{g}$ is the Lie algebra of the Lie group $\mathsf{G}$ and $\mathsf{H}^1_{\rho_0}\left(\Gamma, \mathfrak{g}\right)$ is the group cohomology.
\end{proposition}
\begin{proof}
Using proposition \ref{z1} and lemma \ref{z2} we get that
\[\left.\frac{d}{dt}\right|_{t = 0}\mathsf{b}_{\rho_t}\left(\eta^+,\gamma^-,\gamma^+,\eta^-\right)=0\]
for all $\gamma,\eta\in\Gamma$. Now using the proof of lemma 10.8 of \cite{pressure metric} we deduce that
\[\left[\left.\frac{d}{dt}\right|_{t=0} \mathtt{L}_{\rho_t}\right] = 0\]
in $\mathsf{H}^1_{\mathtt{L}_{\rho_0}}(\Gamma, \mathfrak{so}(2,1))$.
Therefore without loss of generality we can take 
\[\mathtt{L}_{\rho_t} = g_t\mathtt{L}_{\rho_0}g_t^{-1}\] 
for some smooth path $\{g_t\}\subset\mathsf{SO}^0(2,1)$. Now again using proposition \ref{z1} we get that for all $\gamma\in\Gamma$
\[\left.\frac{d}{dt}\right|_{t=0}\left\langle \mathtt{u}_{\rho_t}(\gamma)\mid\nu_{\rho_t}\left(\gamma^-,\gamma^+\right)\right\rangle=0.\]
We notice that $\nu_\rho$ only depends on $\mathtt{L}_\rho$. Hence 
\[\nu_{\rho_t}=g_t\nu_{\rho_0}\] 
for all $t$ and we obtain
\begin{align*}
\left\langle \left.\frac{d}{dt}\right|_{t=0} g_t^{-1}\mathtt{u}_{\rho_t}(\gamma)\mid\nu_{\rho_0}\left(\gamma^-,\gamma^+\right)\right\rangle = 0
\end{align*}
for all $\gamma\in\Gamma$. Now using theorem 1.2 of \cite{cd} we deduce that
\[\left[\left.\frac{d}{dt}\right|_{t=0} g_t^{-1}\mathtt{u}_{\rho_t}\right]=0\]
in $\mathsf{H}^1_{\mathtt{L}_{\rho_0}}(\Gamma, \mathfrak{so}(2,1))$. Hence it follows that
\[\left[\left.\frac{d}{dt}\right|_{t=0} \left(\mathtt{L}_{\rho_t}, \mathtt{u}_{\rho_t}\right)\right]=[\dot{\rho}_0] = 0\] 
in $\mathsf{H}^1_{\rho_0}\left(\Gamma, \mathfrak{g}\right)$.
\end{proof}

\subsection{Margulis multiverse}

Let $h_{\rho}$ be the topological entropy related to a representation $\rho\in\mathsf{Hom}_{\hbox{\tiny $\mathrm{M}$}}(\Gamma, \mathsf{G})$. We recall from equation \ref{hh} that
\begin{align}\label{h}
h_{\rho} = \lim_{T \to \infty}\frac{1}{T} \log \left(\# \left\{[\gamma] \in\mathsf{O}(\Gamma) \mid \alpha_\rho(\gamma) \leqslant T \right\}\right).
\end{align}
Moreover, we also recall that the map
\begin{align}
h : \mathsf{Hom}_{\hbox{\tiny $\mathrm{M}$}}(\Gamma, \mathsf{G}) &\longrightarrow \mathbb{R} \\
\notag \rho &\longmapsto h_{\rho}
\end{align}
is analytic. Now we define the $\textit{constant entropy sections}$ of $\mathsf{Hom}_{\hbox{\tiny $\mathrm{M}$}}(\Gamma, \mathsf{G})$ for any positive real number $k$ as follows:
\begin{align}
\mathsf{Hom}_{\hbox{\tiny $\mathrm{M}$}}(\Gamma, \mathsf{G})_k \defeq \left\{ \rho \in \mathsf{Hom}_{\hbox{\tiny $\mathrm{M}$}}(\Gamma, \mathsf{G}) \mid h_{\rho} = k \right\}.
\end{align} 
We note that if $(\varrho, \mathtt{u})$ is in $\mathsf{Hom}_{\hbox{\tiny $\mathrm{M}$}}(\Gamma, \mathsf{SO}^0(2,1)\ltimes\mathbb{R}^3)=\mathsf{Hom}_{\hbox{\tiny $\mathrm{M}$}}(\Gamma, \mathsf{G})$ then so is $(\varrho, c\mathtt{u})$ where $c$ is some positive real number.
\begin{lemma}\label{h1}
Let $(\varrho, \mathtt{u})$ be in $\mathsf{Hom}_{\hbox{\tiny $\mathrm{M}$}}(\Gamma, \mathsf{SO}^0(2,1)\ltimes\mathbb{R}^3)$ then for any positive real number $c$ we have
\begin{align*}
h_{(\varrho, c\mathtt{u})} = \frac{1}{c}\left. h_{(\varrho, \mathtt{u})}\right. .
\end{align*}
\end{lemma}
\begin{proof}
Using the definition of the Margulis invariant we have that
\begin{align*}
\alpha_{(\varrho, c\mathtt{u})}(\gamma) &= \left\langle c\mathtt{u}(\gamma) \mid \nu_\varrho\left(\gamma\right)\right\rangle \\
&= c \left\langle \mathtt{u}(\gamma) \mid \nu_\varrho\left(\gamma\right)\right\rangle = c \left.\alpha_{(\varrho, \mathtt{u})}(\gamma)\right. .
\end{align*}
where $\nu_\varrho\left(\gamma\right)\defeq\nu_\varrho\left(\gamma^-, \gamma^+\right)$.
Now using equation \ref{h} we get that
\begin{align*}
h_{(\varrho, c\mathtt{u})} &= \lim_{T \to \infty}\frac{1}{T} \log \left(\# \left\{\gamma \in \Gamma \mid \alpha_{(\varrho, c\mathtt{u})}(\gamma) \leqslant T \right\}\right)\\
&= \lim_{T \to \infty}\frac{1}{T} \log \left(\# \left\{\gamma \in \Gamma \mid \alpha_{(\varrho, \mathtt{u})}(\gamma) \leqslant \frac{T}{c} \right\}\right)\\
&= \frac{1}{c}\lim_{T \to \infty}\frac{1}{T} \log \left(\# \left\{\gamma \in \Gamma \mid \alpha_{(\varrho, \mathtt{u})}(\gamma) \leqslant T \right\}\right) = \frac{1}{c}\left. h_{(\varrho, \mathtt{u})}\right. .
\end{align*}
\end{proof}

\begin{lemma}\label{multiverse}
Let $\mathsf{Hom}_{\hbox{\tiny $\mathrm{M}$}}(\Gamma, \mathsf{G})_k$ be a constant entropy section for some real number $k$ then $\mathsf{Hom}_{\hbox{\tiny $\mathrm{M}$}}(\Gamma, \mathsf{G})_k$ is a codimension one analytic submanifold of $\mathsf{Hom}_{\hbox{\tiny $\mathrm{M}$}}(\Gamma, \mathsf{G})$.
\end{lemma}
\begin{proof}
We consider the analytic map $h$ and using lemma \ref{h1} notice that 
\begin{align*}
\left.\frac{d}{dt}\right|_{t = 0}h\left(\varrho, \frac{1}{1+t}\mathtt{u}\right) &= h(\varrho, \mathtt{u}) \left.\frac{d}{dt}\right|_{t = 0} (1+t)\neq 0.
\end{align*}
Hence $\texttt{Rk}(\mathsf{D}h_{(\varrho, \mathtt{u})})=1$. Now using the Implicit function theorem we conclude that $\mathsf{Hom}_{\hbox{\tiny $\mathrm{M}$}}(\Gamma, \mathsf{G})_k = h^{-1}(k)$ is an analytic submanifold of $\mathsf{Hom}_{\hbox{\tiny $\mathrm{M}$}}(\Gamma, \mathsf{G})$ with codimension 1.
\end{proof}
\begin{remark}
The following map
\begin{align*}
{\mathcal{I}}_k : \mathsf{Hom}_{\hbox{\tiny $\mathrm{M}$}}(\Gamma, \mathsf{G})_1 &\longrightarrow \mathsf{Hom}_{\hbox{\tiny $\mathrm{M}$}}(\Gamma, \mathsf{G})_k\\
(\varrho, \mathtt{u}) &\longmapsto \left(\varrho, \frac{1}{k}\mathtt{u}\right)
\end{align*}
gives an analytic isomorphism between $\mathsf{Hom}_{\hbox{\tiny $\mathrm{M}$}}(\Gamma, \mathsf{G})_1$ and $\mathsf{Hom}_{\hbox{\tiny $\mathrm{M}$}}(\Gamma, \mathsf{G})_k$.
\end{remark}
\begin{lemma}
The space $\mathsf{Hom}_{\hbox{\tiny $\mathrm{M}$}}(\Gamma, \mathsf{G})$ is analytically isomorphic to the space $\mathsf{Hom}_{\hbox{\tiny $\mathrm{M}$}}(\Gamma, \mathsf{G})_1 \times \mathbb{R}$.
\end{lemma}
\begin{proof}
We define two analytic maps as follows
\begin{align*}
\mathfrak{h} : \mathsf{Hom}_{\hbox{\tiny $\mathrm{M}$}}(\Gamma, \mathsf{G}) &\longrightarrow \mathsf{Hom}_{\hbox{\tiny $\mathrm{M}$}}(\Gamma, \mathsf{G})_1 \times \mathbb{R}\\
\rho = (\varrho, \mathtt{u}) &\longmapsto \left((\varrho, h_{\rho}\mathtt{u}),h_\rho\right)
\end{align*}
and
\begin{align*}
\mathfrak{h}^\prime : \mathsf{Hom}_{\hbox{\tiny $\mathrm{M}$}}(\Gamma, \mathsf{G})_1 \times \mathbb{R} &\longrightarrow \mathsf{Hom}_{\hbox{\tiny $\mathrm{M}$}}(\Gamma, \mathsf{G})\\
 ((\varrho, \mathtt{u}),r) &\longmapsto \left(\varrho, \frac{1}{r} \mathtt{u}\right).
\end{align*}
We conclude our result by observing that $\mathfrak{h}^\prime \circ \mathfrak{h} = \mathsf{Id}$ and $\mathfrak{h}\circ\mathfrak{h}^\prime = \mathsf{Id}$.
\end{proof}
\begin{definition}\label{multi}
We define the $\textit{Margulis multiverse}$ with entropy $k$ to be 
\[\mathcal{M}_k \defeq \mathsf{Hom}_{\hbox{\tiny $\mathrm{M}$}}(\Gamma, \mathsf{G})_k/_{\sim}\] 
where $k$ is a positive real number and $\rho_1 \sim \rho_2$ if $\rho_1$ is a conjugate of $\rho_2$ by some element of the group $\mathsf{G}=\mathsf{SO}^0(2,1)\ltimes\mathbb{R}^3$.
\end{definition}

\subsection{Riemannian metric on Margulis multiverse}

In this section we finally prove that the pressure metric $\mathtt{P}$ restricted to the constant entropy sections of $\mathsf{Hom}_{\hbox{\tiny $\mathrm{M}$}}(\Gamma, \mathsf{G})$ is Riemannian.

\begin{proof}[Proof of Theorem \ref{mainthm}]
We consider the definition \ref{multi} and observe that the result follows from proposition \ref{final} and lemma \ref{multiverse}.
\end{proof}

\begin{proof}[Proof of Theorem \ref{m2}]
Let $\rho =(\mathtt{L}_\rho,\mathtt{u}_\rho)$ be a point in $\mathsf{Hom}_{\hbox{\tiny $\mathrm{M}$}}(\Gamma, \mathsf{G})$ and for $\epsilon>0$ let  
\[\{\rho_t\defeq(\mathtt{L}_{\rho},(1+t)\mathtt{u}_{\rho})\}_{t\in(-\epsilon,\epsilon)}\] be a smooth path in $\mathsf{Hom}_{\hbox{\tiny $\mathrm{M}$}}(\Gamma, \mathsf{G})$. We notice that if $\mathsf{f}_0$ is a reparametrization coming from $\rho$ then 
\[\mathsf{f}_t\defeq(1+t)\mathsf{f}_0\] is a reparametrization which comes from $\rho_t$. We also notice that the entropy
\[h_{\rho_t}=\frac{h_\rho}{1+t}.\]
Therefore we get
\[\left.\frac{d}{dt}\right|_{t=0}h_{\rho_t}\mathsf{f}_{\rho_t}=\left.\frac{d}{dt}\right|_{t=0}h_{\rho}\mathsf{f}_{\rho}=0.\]
Hence by proposition \ref{p} we get that $\mathtt{P}(\dot{\rho}_0,\dot{\rho}_0)=0$ where $\dot{\rho}_0\defeq\left.\frac{d}{dt}\right|_{t=0}\rho_t$ and $[\dot{\rho}_0]\neq0$ in $\mathsf{H}^1_{\rho_0}\left(\Gamma, \mathfrak{g}\right)$. Now using remark \ref{nonneg} we conclude that $\mathtt{P}$ has signature $(\dim(\mathcal{M})-1,0)$ over the moduli space $\mathcal{M}$.
\end{proof}

\end{document}